\documentclass[a4ppaper]{amsart}

\usepackage{enumerate}
\usepackage[latin1]{inputenc}
\usepackage{amsmath}
\usepackage{amsthm}
\usepackage{latexsym}
\usepackage{amssymb}
\usepackage[left=2.61cm,right=2.61cm,top=2.72cm,bottom=2.72cm]{geometry}
\usepackage[hyperpageref]{backref}
\usepackage{hyperref}

\usepackage{color}
\usepackage{changes}
\setremarkmarkup{(#2)}
\usepackage[all]{xy}
\usepackage{calc}
\usepackage{multicol}

\theoremstyle{plain}
\newtheorem{thm}{Theorem}[section]
\newtheorem{prop}[thm]{Proposition}
\newtheorem{defi}{Definition}[section]
\newtheorem{lem}[thm]{Lemma}
\newtheorem{cor}[thm]{Corollary}

\newtheorem{rmq}[thm]{Remark}

\newcommand{\R}{\mathbb{R}}
\numberwithin{equation}{section}
\newcommand{\N}{\mathbb{N}}

\newcommand{\Cc}{\mathbb{C}}
\newcommand{\U}{\mathcal{U}}
\newcommand{\Lb}{\mathcal{L}}
\newcommand{\Ic}{\mathcal{I}}
\newcommand{\Lc}{\mathbb{L}}
\newcommand{\Hc}{\mathbb{H}}
\newcommand{\Rc}{\mathcal{R}}
\newcommand{\loc}{\text{loc}}

\begin{document}\title[Mixed 4NLS]{Orbitally stable standing waves of a mixed dispersion nonlinear Schr\"odinger equation}

\author[Bonheure, Casteras, Moreira dos Santos  and Nascimento]{Denis Bonheure \and Jean-Baptiste Casteras \and Ederson Moreira dos Santos \and Robson Nascimento}

\address{Denis Bonheure, Jean-Baptiste Casteras \and Robson Nascimento
\newline \indent D\'epartement de Math\'ematiques, Universit\'e Libre de Bruxelles,
\newline \indent CP 214, Boulevard du triomphe, B-1050 Bruxelles, Belgium,
\newline \indent and INRIA- team MEPHYSTO.}
\email{Denis.Bonheure@ulb.ac.be}
\email{jeanbaptiste.casteras@gmail.com}
\email{Robson.Nascimento@ulb.ac.be}

\address{Ederson Moreira dos Santos 
\newline \indent Instituto de Ci{\^e}ncias Matem{\'a}ticas e de Computa\c{c}{\~a}o --- Universidade de S{\~a}o Paulo \newline \indent
Caixa Postal 668, CEP 13560-970 - S\~ao Carlos - SP - Brazil}
\email{ederson@icmc.usp.br}

\begin{abstract}
We study the mixed dispersion fourth order nonlinear Schr\" odinger equation
\begin{equation*}
i \partial_t \psi -\gamma \Delta^2 \psi +\beta \Delta \psi +|\psi|^{2\sigma} \psi =0\ \text{in}\ \R \times\R^N,
\end{equation*}
where $\gamma,\sigma>0$ and $\beta \in \R$. We focus on standing wave solutions, namely solutions of the form $\psi (x,t)=e^{i\alpha t}u(x)$, for some $\alpha \in \R$. This ansatz yields the fourth-order elliptic equation
\begin{equation*}
\gamma \Delta^2 u -\beta \Delta u +\alpha u =|u|^{2\sigma} u.
\end{equation*}
We consider two associated constrained minimization problems: one with a constraint on the $L^2$-norm and the other on the $L^{2\sigma +2}$-norm. Under suitable conditions, we establish existence of minimizers and we investigate their qualitative properties, namely their sign, symmetry and decay at infinity as well as their uniqueness, nondegeneracy and orbital stability. 
\end{abstract}

\maketitle

\tableofcontents

\section{Introduction}

The canonical model for propagation of intense laser beams in a bulk $2d$-medium with Kerr nonlinearity is given by the cubic nonlinear Schr\"odinger equation
\begin{equation}
\label{2nls}
i\partial_t \psi +\Delta \psi +|\psi|^2 \psi=0,\ \psi (0,x)=\psi_0 (x),\ (t,x)\in \R\times \R^2.
\end{equation}
This equation is a particular case of the Schr\"odinger equation in arbitrary dimension with a general pure power nonlinearity
\begin{equation}
\label{nlsgen}
i\partial_t \psi +\Delta \psi +|\psi|^{2\sigma} \psi=0,\ \psi (0,x)=\psi_0 (x),\ (t,x)\in \R\times\R^N,
\end{equation}
where $\sigma$ is a given positive real number. The equation \eqref{nlsgen} is one of the most studied PDE and its importance in mathematical physics is a source of many interesting mathematical problems. It is well known that global existence in time and stability of standing waves (in nonlinear optics, those are usually called waveguides) of \eqref{nlsgen}, i.e. solutions of the form $\psi (t,x)=e^{i \alpha t}u(x)$ for some $\alpha\in \R$, depend on the sign of $\sigma N -2$. When $\sigma N<2$, all solutions to \eqref{nlsgen} exist globally in time and standing waves are orbitally stable (see Definition \ref{defiorb}), whereas if $\sigma N\geq 2$, then finite time blow-up may occur and the waveguide solutions become unstable. We refer for instance to \cite{Weinstein83,Caz,SulemSulem}.
This means $\sigma N=2$ is critical and depending on the point of view, one either says $\sigma=2/N$ is a critical exponent in dimension $N$ or dimension $2/\sigma$ is critical for the nonlinear Schr\"odinger equation \eqref{nlsgen}. Observe that for $N=2$ and $N=3$, the Kerr nonlinearity is respectively critical and supercritical. 
 
In order to regularize and stabilize the solutions to \eqref{nlsgen}, Karpman and later Karpman and Shagalov, see \cite{MR1779828} and the references therein, have proposed to add higher order dispersive terms in the model as an alternative to stabilization by saturation of the nonlinearity, see for instance \cite{LSSJMP,PhysRevA.30.3279}. Namely they considered the equation
\begin{equation}
\label{4nlsdis}
i \partial_t \psi -\gamma \Delta^2 \psi +\beta\Delta \psi +|\psi|^{2\sigma} \psi =0, 
\end{equation}
for some $\gamma >0$ which is assumed to be small, and with $\beta=1$. Thanks to this fourth order term, using a combination of stability analysis and numerical simulations, they showed that when $0<N\sigma <4$ and ($\gamma$ is small enough if $2\leq N\sigma < 4$), waveguide solutions are stable and when $N\sigma>4$, they become unstable. Their results highlight the existence of a second critical value $N\sigma=4$, which results from the presence of the biharmonic term. We observe that the Kerr nonlinearity is now subcritical in dimension $2$ and $3$ in this extended model. A phenomenological motivation for the introduction of the small fourth-order dissipation term is given in \cite{MR1898529}. In nonlinear optics, \eqref{2nls} is derived from the scalar nonlinear Helmhotz equation through the so-called paraxial approximation \cite{MR1898529}.  The fact that solutions to \eqref{2nls} can blow-up at finite-time suggests that some terms neglected in the paraxial approximation should prevent the blow-up. Since a small biharmonic term arises as (part of) the nonparaxial correction to NLS, it is natural to consider this term as small but nonzero and study its effect on the blow-up.  The authors of \cite{MR1898529} show how the new critical threshold $N\sigma=4$ arises. By extending the approach of Weinstein \cite{Weinstein83}, they prove that when $\sigma N<4$, all solutions to \eqref{4nlsdis} exist globally in time. On the other hand, they mention that existence of blowing-up solutions for $\sigma N\ge4$ is a difficult open problem (which has now been recently partially solved in \cite{boulengerlenzmann}).

The extended model \eqref{4nlsdis} has attracted less attention than its classical counterpart \eqref{nlsgen} though with an increasing interest more recently. One should distinguish \eqref{4nlsdis} often called mixed dispersion NLS, from the biharmonic NLS or $4$NLS
\begin{equation}
\label{4nls}
i \partial_t \psi -\gamma \Delta^2 \psi +|\psi|^{2\sigma} \psi =0,
\end{equation}
which might behave quite differently. In \eqref{4nlsdis}, the driving dispersive term is the Laplacian when $\gamma$ is small while the biharmonic NLS corresponds to the limit case $\gamma\to\infty$ (after a rescaling in the variable $x$, as that performed to relate \eqref{eq4gamma} to \eqref{eq4beta}). Both models have been considered in \cite{MR1898529} where various properties of the equations are described, relying in part on numerical results. We refer to the works of Ben-Artzi, Koch and Saut \cite{MR1745182} and Pausader \cite{MR2353631,MR2502523,MR2505703} for well-posedness and scattering, see also \cite{MR2515176,MR2746203,MR3078112,MR3462127} and to the recent work of Boulenger and Lenzmann \cite{boulengerlenzmann} and the references therein concerning finite-time blow-up. We also mention that the one-dimensional stationary mixed dispersion NLS arises in the theory of water waves \cite{Buffoni1995109,Buffoni1996}.

\medbreak

In this paper, we focus on standing wave solutions to \eqref{4nlsdis}. 
The ansatz $\psi (t,x)=e^{i \alpha t}u(x)$ yields the fourth-order semilinear elliptic equation 
\begin{equation}\label{eq4gamma}
\gamma \Delta^2 u - \beta\Delta u +\alpha u= |u|^{2\sigma} u\ \text{in}\ \R^N.
\end{equation}
Setting  $v(x)=u(\gamma^{\frac{1}{4}}x)$ with $\theta=\frac{\beta}{\sqrt{\gamma}}$, we see that $v$ solves 
\begin{equation}\label{eq4beta}
\Delta^2 v -\theta \Delta v+\alpha v=|v|^{2\sigma}v\ \text{in}\ \R^N .
\end{equation}
Observe that the effect of this change of variables on the $L^2$-norm is given by $\|v\|^2_{L^2}= {\gamma^{-\frac{N}{4}}}\|u\|^2_{L^2}$. This relation will be relevant in subsequent discussions. 

\medbreak

We deal simultaneously with one or the other of those equations \eqref{eq4gamma} and \eqref{eq4beta} except in Section $4$, where we investigate the case when $\gamma \rightarrow 0$ which models the effect of a small fourth order dispersion (or the effect of a large second order dispersion as seen by scaling). More precisely, we consider two constrained minimization problems which naturally arise as for \eqref{nlsgen}. Namely, since the $L^2$-norm is conserved along the flow for \eqref{4nlsdis}, it is natural to look first for standing waves having a prescribed $L^2$-norm. Such solutions were built by Cazenave and Lions \cite{Caz} for \eqref{nlsgen}. 
Their construction consists in minimizing the functional $E_{0}:H^1(\R^N)\to \R$ defined by
\begin{equation}\label{E_0}
 {E}_{0}(u)=\frac{1}{2}\int_{\R^N}|\nabla u|^2\, dx-\frac{1}{2\sigma+2}\int_{\R^N}|u|^{2\sigma+2}\, dx
\end{equation} 
under the constraint $\|u\|_{L^2}^2=\mu$. If $0<N\sigma<2$, ${E}_0$ achieves its infimum and any associated minimizer solves 
\begin{equation}\label{nlsalpha}
- \Delta u +\alpha u= |u|^{2\sigma} u\ \text{in}\ \R^N,
\end{equation}
with the Lagrange multiplier 
\begin{equation}
\label{lagmult2nls}
\alpha = \frac{1}{\mu}\left(\int_{\R^N}|u|^{2\sigma+2}\, dx - \int_{\R^N}|\nabla u|^2\, dx\right) = 
\frac{1}{\mu}\left(\dfrac{(2-\sigma N) + 2\sigma}{2\sigma +2}\right)\int_{\R^N}|u|^{2\sigma +2}dx,
\end{equation}
where the last equality follows from the Derrick-Pohozahev identity \cite{MR0174304,MR0192184}. Moreover, Cazenave and Lions  \cite[Theorem II.2]{Caz} showed that those standing waves minimizing $E_0$ are orbitally stable for \eqref{nlsgen} whereas standing waves built for instance in \cite{MR695535,MR695536} are unstable for $2/N<\sigma<2/(N-2)$ as arbitrarily close initial conditions lead to blowing up solutions, see \cite[Remark II.2]{Caz} .  

\medbreak

For \eqref{4nlsdis}, we obtain the following counterpart. Define
\begin{equation}\label{MinL2fixed}
 I_{\gamma,\beta}({\mu})=\inf_{u\in M_{\mu}}E_{\gamma,\beta}(u)
\end{equation}
where
\begin{equation}\label{de:Mmu}
 M_{\mu}=\{u\in H^2(\R^N):\int_{\R^N}|u|^{2}\, dx=\mu\}
\end{equation}
and
\begin{equation}\label{def:Egammabeta}
 E_{\gamma,\beta}(u)=\frac{\gamma}{2}\int_{\R^N}|\Delta u|^2\, dx+\frac{\beta}{2}\int_{\R^N}|\nabla u|^2\, dx-\frac{1}{2\sigma+2}\int_{\R^N}|u|^{2\sigma+2}\, dx.
\end{equation}
If $I_{\gamma,\beta}(\mu)$ is achieved, then any associated minimizer solves 
\eqref{eq4gamma}
with the Lagrange multiplier 
\begin{align}
\label{defalphaintro}
\alpha  & =  \frac{1}{\mu}\left(\int_{\R^N}|u|^{2\sigma+2}\, dx - {\gamma}\int_{\R^N}|\Delta u|^2\, dx-{\beta}\int_{\R^N}|\nabla u|^2\, dx \right)\nonumber \\
& =  \frac{1}{\mu}\left(-2 E_{\gamma,\beta}(u) + \dfrac{\sigma}{\sigma+1}\int_{\R^N}  |u|^{2\sigma +2}\, dx\right).
\end{align}
This implies $\alpha>0$ if $E_{\gamma,\beta}(u)<0$. The next theorem is our main result concerning this minimization problem.

\begin{thm}\label{Compact-Min-Sol}
Assume $\gamma>0$ and $\beta\ge 0$. If $0<\sigma<2/N$, then $ I_{\gamma,\beta}({\mu})$ is achieved for every $\mu>0$. If ${2}/{N}\le \sigma<{4}/{N}$, then there exists a critical mass $\mu_c(\gamma,\beta,\sigma)$ such that 
\begin{enumerate}[(i)]
\item $ I_{\gamma,\beta}({\mu})$ is not achieved if $\mu< \mu_c$;
\item $ I_{\gamma,\beta}({\mu})$ is achieved if $\mu > \mu_c$ and $\sigma = 2/N$;
\item $ I_{\gamma,\beta}({\mu})$ is achieved if $\mu \ge \mu_c$ and $\sigma \neq 2/N$;
\end{enumerate}
\begin{equation}\label{lim-mu}
\lim_{\gamma\to 0} \mu_c(\gamma,\beta,\sigma) =  
0
\end{equation}
and 
\begin{equation}
\lim_{\beta\to 0} \mu_c(\gamma,\beta,\sigma) 
= \mu_c(\gamma,0,\sigma) = 0.
\end{equation}
If moreover $2\sigma \in \N$ and $I_{\gamma,\beta}({\mu})$ is achieved, then there exists at least one radially symmetric minimizer. Finally, if $\sigma> 4/N$, then $ I_{\gamma,\beta}({\mu})= -\infty$ for every $\mu>0$.
\end{thm}


To the best of our knowledge, this is the first result in the literature concerning the existence of standing waves of \eqref{4nlsdis} with a prescribed $L^2$-mass. Observe that for \eqref{4nls}, the result mirrors that for \eqref{nlsgen}, since $\mu_c(\gamma,0,\sigma) = 0$, whereas there is a critical mass for existence in the case of \eqref{4nlsdis}. This is due to a lack of homogeneity when $\beta>0$. Indeed, all the terms of the functional to be minimized scale differently. Such a behaviour is present in other models like the Schr\"odinger-Poisson equation, see \cite{dolb,jeanjean}. One also notices that for $N=2$ and $N=3$ ($\sigma=1$), the Kerr nonlinearty which is the more relevant case in optics is subcritical for both \eqref{4nls} and \eqref{4nlsdis}. Let us also emphisize that if $\sigma> 4/N$, then $ I_{\gamma,\beta}({\mu})$ cannot be achieved since $ I_{\gamma,\beta}({\mu})= -\infty$ for every $\mu>0$. Existence results of solutions with prescribed mass when $\sigma \geq 4/N$ are given in \cite{BCGJ} by the first two authors in collaboration with Gou and Jeanjean.

\medbreak

We now turn to the second natural variational problem associated with \eqref{4nls} and \eqref{4nlsdis}. Indeed, if one looks for time independent solutions, it is natural to ask whether there exists a stationary solution which minimizes the action functional $A:H^2(\R^N)\to \R$ defined by 
\begin{equation}
\label{A-funct}
A(u)= \frac{1}{2}J_{\gamma,\beta,\alpha}(u) - \frac{1}{2\sigma+2 }\int_{\R^N} |u|^{2\sigma+2} dx,
\end{equation}
where $J_{\gamma,\beta,\alpha}$ is the quadratic form defined by 
\begin{equation}
\label{J-quadr}
J_{\gamma,\beta,\alpha}(u)= \gamma\int_{\R^N} |\Delta u|^2\, dx+\beta \int_{\R^N}|\nabla u|^2\, dx+\alpha \int_{\R^N}|u|^2\, dx
\end{equation}
on $H^2(\R^N)$. Observe that when $\alpha,\gamma>0$ and $\beta>-2\sqrt{\gamma \alpha}$, we have that $J_{\gamma,\beta,\alpha}$ is the square of a norm on $H^2(\R^N)$.
A \emph{ground state} is then a least energy critical point of $A$ and therefore a stationary solution which minimizes the action within the set of nontrivial (nonzero) solutions. It is standard to check that if
\begin{equation}\label{MinLpintro}
m=\inf_{u \in M}J_{\gamma,\beta,\alpha} (u),
\end{equation}
where 
\begin{equation}
\label{lpconsint}
M= \{u\in H^2 (\R^N):\int_{\R^N} |u|^{2\sigma +2}dx=1\},
\end{equation}
is achieved by some $u\in M$, then $v=m^{\frac{1}{2\sigma}}u$ is a least energy critical point of $A$.
The following result is proved in \cite{bonnas}. 
\begin{thm}[{\cite[Theorem 1.1]{bonnas}}]
\label{thm1.1bonnas}
Assume $\alpha,\gamma, \sigma>0$, $\beta>-2\sqrt{\gamma \alpha}$ and $\sigma < 4/(N-4)$ if $N\ge 5$. Then problem \eqref{MinLpintro} has a ground state. If $\beta\ge 2\sqrt{\gamma\alpha}$, then any ground state $u$ is such that $|u|$ is positive, radially symmetric around some point and strictly radially decreasing.  
\end{thm}

Those qualitative properties of ground states are well-known for $\gamma=0$. In addition, it is well-known \cite{Kwong} that $\psi(t,x) = \exp(i\alpha t)v(x)$ gives a standing wave solution to \eqref{nlsgen} which is a ground state if and only if there exists $x_0\in\R^N$ such that
$$\alpha^{-\frac{1}{2\sigma}}v(\alpha^{-\frac12}x) = u(x+x_0),$$
where $u$ is the unique, up to translation, positive $H^1$-solution to
\begin{equation}\label{nlsunique}
- \Delta u +u= |u|^{2\sigma} u\ \text{in}\ \R^N.
\end{equation} 
Obviously $u$ is radially symmetric. From a phenomenological point of view, it is important to understand if the standing waves of \eqref{4nlsdis} differ qualitatively from this (essentially) unique standing wave of \eqref{nlsgen}. Moreover, the main concern of Karpman and Shagalov \cite{MR1779828} or in the small nonparaxial correction in nonlinear optics is to understand this when the fourth order dissipation coefficient $\gamma$ is small. The effect of a small fourth order perturbation on ground states has been considered in \cite{bonnas} when the problem is $H^1$ subcritical under some restriction on the dimension and the power nonlinearity. Here we complete this study for any $H^1$ subcritical power and any dimension and we extend it to minimizers with a prescribed mass. Before stating our results, we recall that a solution $u$ of \eqref{eq4gamma}
is nondegenerate in $H^2(\R^N)$ if for any solution $v$ of the linearized equation
\begin{equation}\label{4nlsgamma-beta-alpha-linearized}
\gamma\Delta^2 v-\beta \Delta v +\alpha v=(2\sigma +1) |u|^{2\sigma}v,
\end{equation} 
there exists $\xi\in \R^N$ such that $v(x) = \xi\cdot\nabla u(x)$. In other words, the kernel of the linearized operator defined by
\begin{equation}\label{L-operator-Linear}
Lv=\gamma\Delta^2 v-\beta \Delta v +\alpha v-(2\sigma +1) |u|^{2\sigma}v
\end{equation}
is
\begin{equation*}
\ker L=\text{span}\{\partial_{x_1}u,\cdots,\partial_{x_N}u\}.
\end{equation*}
\begin{thm}
Assume $0<\sigma<2/N$ and $\beta,\mu>0$. Then there exists $\gamma_0>0$ such that if $0<\gamma<\gamma_0$, then \eqref{MinL2fixed} is achieved by a unique minimizer (up to translations and multiplication by $-1$). This minimizer is a nondegenerate solution to \eqref{eq4gamma}. Fixing its positive maximum at the origin, this solution is radially symmetric and strictly radially decreasing. \end{thm}
\medbreak
 
As already mentioned, this theorem is more general than that of \cite{bonnas} for ground states (see also Theorem \ref{limitlpfixed}). Let us point out that these results allow us to show that solutions with prescribed mass obtained in Theorem \ref{Compact-Min-Sol} are ground state solutions (for the frequency corresponding to the Lagrange multiplier) at least when $0<\sigma<2/N$ and $\gamma$ is small enough. This property holds probably without the smallness assumption on $\gamma$ but we have to leave it as a conjecture only. 

\medbreak

Our last aim in this paper is to investigate the orbital stability of both minimizers with prescribed mass and the ground states. Let us recall that, by \cite[Corollary 4.1]{MR2353631}, global existence in time for the initial value problem associated to \eqref{4nlsdis} is granted provided that $0 <\sigma < 4/N$ if $N>4$. As mentioned above, the standing waves to \eqref{nlsgen} obtained through the $L^2$-constrained minimization of the energy are automatically orbitally stable \cite[Theorem II.2]{Caz} while this is not the case for the standing waves obtained as ground state for $2/N<\sigma<2/(N-2)$ \cite[Remark II.2]{Caz}.\label{page-C-L} The same holds true for \eqref{4nlsdis}. Following closely the arguments of Natali and Pastor \cite{natalipastor}, we also prove that orbital stability holds for ground states $u$ as soon as they are nondegenerate and the following condition holds
\begin{equation}
\label{extracond}
if\ v\in H^2 (\R^N)\ is\ a \ solution\ to\ L v=u,\ then\ \int_{\R^N} v u\, dx < 0,  
\end{equation}
where $L$ is defined in \eqref{L-operator-Linear}. Let us point out that Albert \cite{MR1151253} proved numerically that this condition holds if $\beta=1$, $\alpha=4/25$, $N=1$ and $\sigma=1$. 
\begin{thm}\label{thm:orbit}Let $0<\sigma<4/N$. 
\begin{enumerate}
\item The set $\{ U\in H^2 (\R^N): \text{ $U$ is a solution to } \eqref{MinL2fixed}\}$ is stable (see Definition \ref{defiGorb}).
\item Suppose that $u$ is a  nondegenerate minimizer of \eqref{MinLpintro} satisfying \eqref{extracond}. Then the standing wave $\psi(t,x) = \exp(i\alpha t)u(x)$
is orbitally stable (see Definition \ref{defiorb}).
\end{enumerate}
\end{thm}
\medbreak
It is shown in \cite{amick} that there exists a unique radial solution $u$ to \eqref{eq4gamma} when $N=1$ provided $\beta\geq 2\sqrt{\gamma\alpha}$. Following closely the argument of \cite{buffonitoland}, it is simple to prove that this solution is nondegenerate. As a consequence, we obtain the orbital stability of the ground states in dimension $1$ for any $\alpha$, $\beta$ and $\gamma$ satisfying $\beta\geq 2\sqrt{\gamma\alpha}$ provided that \eqref{extracond} holds. Let us also point out that depending on whether $u$ is a minimizer for \eqref{MinL2fixed} or $u$ is a  nondegenerate minimizer of \eqref{MinLpintro}, we are able to cover different situations. For instance, consider the case $N=1$. As already said, in this case, we know that ground states are orbitally stable provided that $\beta\geq 2\sqrt{\gamma\alpha}$. On the other hand, if $u$ is a minimizer for \eqref{MinL2fixed}, we know that $u$ is a solution to \eqref{eq4gamma} with a Lagrange multiplier $\alpha$ defined as in \eqref{defalphaintro}. It is not clear that the Lagrange multipliers corresponding to admissible masses (for existence of a minimizer) cover the whole range $[0,\beta^2/(4 \gamma)]$.
  
\medbreak

Our manuscript is organized as follows. In Section \ref{sec:eximass}, we consider the minimization problem \eqref{MinL2fixed}. Section \ref{sec:quali} is dedicated to qualitative properties of minimizers for \eqref{MinL2fixed} and ground states. For instance, we show that we can build positive standing waves with a prescribed mass and study when there exist radial minimal standing waves. We also give an alternative proof to the radial symmetry of ground states, when $0<\sigma < 4/(N-4)$ (if $N\geq 5$) and $\beta> 2\sqrt{\gamma\alpha}$, using a purely variational argument. 
We then prove the exponential decay of $H^2$ solutions to \eqref{eq4gamma}. We conclude Section \ref{sec:quali} by proving the nondegeneracy of the unique solution to \eqref{eq4gamma} when $N=1$ and $\beta \geq 2\sqrt{\gamma \alpha}$. In Section \ref{sec:smalldis}, we show the $H^2$ convergence of minimizers of \eqref{MinL2fixed} as $\gamma \rightarrow 0$ to $w_\mu (\beta^{-\frac12}(x+x_0))$ where $w_\mu$ is the unique minimizer of $E_0$ under the constraint $\|w_\mu\|_{L^2}^2=\mu$ with $\max_{x\in\R^N}w_\mu (x)=w_\mu (x_0)$.
This convergence ensures uniqueness (up to translations), nondegeneracy and radial symmetry provided $\gamma$ is small enough. Finally, in Section \ref{sec:orbstab}, we prove our results concerning orbital stability, namely Theorem \ref{thm:orbit}.

\medbreak

To conlude, we would like to mention some questions that we left open. First, it would be interesting to show that the solution obtained in Theorem \ref{thm1.1bonnas} is unique and nondegenerate provided that $\beta \geq 2\sqrt{\gamma \alpha}$ as it is suggested by Theorem \ref{thmnondegen} when $N=1$. We also conjecture that the solutions obtained in Theorem \ref{Compact-Min-Sol} are radially symmetric. 

Concerning stability,  we believe that it should be possible to prove that \eqref{extracond} is satisfied by any minimizer of \eqref{MinLpintro} provided that $\sigma < 4/N$. Notice that when $\sigma \geq 4/N$, the first two author, in collaboration with Gou and Jeanjean proved the instability of ground-states by finite or infinite time blow-up, see \cite{BCGJ1}.

\medbreak

\section{Existence of standing waves with a prescribed mass}\label{sec:eximass}
In this section we study the minimization problem \eqref{MinL2fixed}. By scaling, we can assume either $\gamma=1$ if $\gamma>0$ or $\beta=1$ if $\beta\ne0$. We assume $\gamma>0$ and we fix $\beta=1$. We briefly comment on the case $\beta=0$ below (see Remark \ref{Neg-Minrmq}). 

\medbreak

We begin this section by recalling and proving some Gagliardo-Nirenberg interpolation inequalities. Thanks to them, we are able to estimate the energy functional $E_{\gamma,1}$. In particular, we establish its strict negativity when the mass $\mu$ is sufficiently large. When the energy is strictly negative, which is always the case when $0<\sigma < 2/N$, using the results of \cite{Li84}, we prove the pre-compactness of sequences of minimizers of \eqref{MinL2fixed}, which leads to our existence results. When the energy is not strictly negative, we obtain non-existence results provided the mass is strictly less than a certain threshold value. In the limit case and when $\sigma \neq 2/N$, we are able to prove the non-vanishing of sequences of minimizers which leads to the existence of a minimizer. The case $\sigma = 2/N$ is left open. We conclude the section by showing that at least one minimizer of \eqref{MinL2fixed} is radially symmetric whenever $2\sigma \in\N$. 

\subsection{Gagliardo-Nirenberg interpolation inequalities}

We begin by recalling two well-known Gagliardo-Nirenberg interpolation inequalities for functions $u\in H^2(\R^N)$, namely
\begin{equation}\label{G-N-H2-ineq}
 \|u\|^{2\sigma+2}_{L^{2\sigma+2}}\leq B_N(\sigma)\|\Delta u\|^{\frac{\sigma N}{2}}_{L^2}\|u\|^{2+2\sigma-\frac{\sigma N}{2}}_{L^2},
\end{equation}
where 
\begin{align*}
     \begin{cases}
     0\leq\sigma,& \text{for}\ N\leq 4,\\
     0\leq\sigma\leq\dfrac{4}{N-4},& \text{for}\ N>4,
     \end{cases} 
\end{align*}
and
\begin{equation}\label{G-N-H1-ineq}
 \|u\|^{2\sigma+2}_{L^{2\sigma+2}}\leq C_N(\sigma)\|\nabla u\|^{\sigma N}_{L^2}\|u\|^{2+\sigma(2-N)}_{L^2},
\end{equation}
where 
\begin{align*}
     \begin{cases}
     0\leq\sigma,& \text{for}\ N\leq 2,\\
     0\leq\sigma\leq\dfrac{2}{N-2},& \text{for}\ N>2.     
     \end{cases} 
  \end{align*}
See for instance \cite{Ga-57,Ga-58,Nir-59}. The constants $B_N(\sigma)$ and $C_N(\sigma)$ depend on $\sigma$ and $N$. 
Thanks to these inequalities, we can prove a $2$-parameters Gagliardo-Nirenberg interpolation type inequality involving the $L^2$-norms of $u,\nabla u$ and $\Delta u$. 
\begin{lem}\label{lem:GNI4terms}
Assume $\sigma>0$ and $\sigma < 4/(N-4)$ if $N>4$. Let $0<\delta < \sigma < \tau$ and assume $\tau<4/(N-4)$ if $N>4$ and $\delta<2/(N-2)$ if $N>2$. Then, there exists $C>0$ such that  
\begin{equation}
\label{GNtriplegen}
\int_{\R^N} |u|^{2\sigma +2}\, dx \leq C \left(\int_{\R^N} u^2\, dx \right)^{p} \left(\int_{\R^N} |\nabla u|^2\,  dx\right)^{q}
\left(\int_{\R^N} |\Delta u|^2\, dx\right)^{r},
\end{equation}
for all $u\in H^2(\R^N)$, where $p={1-\frac{\sigma (N-4) + N\delta (1-\lambda)}{4}}$, $q=\frac{\delta N}{2}(1-\lambda)$, $r=\frac{\tau N}{4} \lambda$ and 
$\lambda =(\sigma -\delta)/(\tau - \delta)$. Moreover, we have 
$C\leq (B_N (\tau))^\lambda ( C_N (\delta))^{1-\lambda}.$
\end{lem}
\begin{proof}
We write $\sigma=\lambda \tau +(1-\lambda)\delta$. Using H\" older inequality, we have
\begin{equation}\label{eq:triplel2}
\int_{\R^N} |u|^{2\sigma +2}dx \leq  \left(\int_{\R^N}| u|^{2\tau +2} dx \right)^\lambda \left(\int_{\R^N} | u|^{2\delta +2} dx\right)^{1-\lambda}.
\end{equation}
Then, the result follows by applying Gagliardo-Nirenberg inequalities \eqref{G-N-H2-ineq} and \eqref{G-N-H1-ineq} to the right-hand side of \eqref{eq:triplel2}. 
\end{proof}

\begin{cor}\label{newcor}
Let $2/N<\sigma<4/N$. Then there exists a constant $C_{\sigma,N}>0$ such that
\begin{equation}
\label{GNtriple}
\int_{\R^N} |u|^{2\sigma +2}\, dx \leq C_{\sigma,N} \left(\int_{\R^N} u^2\, dx \right)^\sigma \left(\int_{\R^N} |\nabla u|^2\, dx\right)^{\frac{4-\sigma N}{2}}\left(\int_{\R^N} |\Delta u|^2\, dx\right)^{\frac{\sigma N}{2} -1} ,
\end{equation}
for all $u \in H^2(\R^N)$.
\end{cor}
\begin{proof}
Since $2/N<\sigma<4/N$, we infer that $0<\frac{4 -\sigma N}{N} < \frac{2}{N}$. Set 
\begin{equation}\label{eq:tau}
\tau = \frac{2\sigma N -4}{N + (\sigma N -4)\delta^{-1}}
\end{equation}
with $\frac{4 -\sigma N}{N} < \delta < \sigma$ and the extra condition $\delta < \frac{2}{N-2}$ if $N \geq3$ (recall that $\frac{4 -\sigma N}{N} < \frac{2}{N}$). Then observe that $\delta < \sigma$ implies $\tau > \sigma$. So far, within the mentioned conditions, we have
\[
0 < \delta < \sigma < \tau \quad \text{with $\delta < \frac{2}{N-2}$ if $N \geq3$}.
\]
Next we will guarantee that it is possible to find $\delta$ and $\tau$ as above such that $\tau < \frac{4}{N-4}$ if $N \geq 5$. Indeed, with $N \geq 5$ we have $\frac{2}{N-2} < \frac{4}{N}$ and $\delta < \frac{2}{N-2}$ implies $\tau < \frac{4}{N-4}$, with $\tau$ as in \eqref{eq:tau}. Hence, with $\tau$ given by \eqref{eq:tau}, we have shown that it is possible to find
\[
0 < \delta < \sigma < \tau \quad \text{with $\delta < \frac{2}{N-2}$ if $N \geq3$ and $\tau < \frac{4}{N-4}$ if $N \geq 5$}.
\]
Finally, with $\tau$ given by \eqref{eq:tau}, we may solve the system of equations
\[ 
1-\frac{\sigma (N-4) + N\delta (1-\lambda)}{4} = \sigma, \quad
\frac{\delta N}{2}(1-\lambda) = \frac{4 -\sigma N}{2}, \quad
\frac{\tau N}{4} \lambda = \frac{\sigma N}{2} -1, \ \  
\text{with} \ \ \lambda =(\sigma -\delta)/(\tau - \delta),
\]
and hence \eqref{GNtriple} follows from \eqref{GNtriplegen}. Indeed, in the system above, the first and the second equations are equivalent and hence we just need to solve the second and third, and these can be rewritten as
\[
\delta N \lambda = \delta N + \sigma N -4 \ \ \text{and} \ \ \tau N \lambda = 2 \sigma N -4,
\]
which are solvable and induce the formula \eqref{eq:tau} for $\tau$.
\end{proof}

\subsection{Estimates of the energy}
This subsection is devoted to energy estimates on the functional $E_{\gamma,1}$. The main aim is to deduce the sign of $I_{\gamma,1}({\mu})$ as a function of $\mu$. We begin by showing the coercivity of $E_{\gamma,1}$ when $\sigma < 4/N$. We recall that $M_{\mu}$ is defined in \eqref{de:Mmu} and that
\[
u \mapsto \left( \int_{\R^N} (\Delta u)^2 + u^2 \, dx\right)^{1/2}
\]
is a norm on $H^2(\R^N)$ which is equivalent to the usual one
\[
u \mapsto \left( \int_{\R^N} |D^2 u|^2 + |\nabla u|^2 + u^2 \, dx\right)^{1/2}.
\]
\begin{lem}\label{Lem-Min}
The energy $E_{\gamma,1}$ is bounded from below and coercive over $M_{\mu}$ when $0<\sigma<4/N$. Moreover, for $\sigma\in (0,4/N)$ the map $\mu\mapsto I_{\gamma,1}({\mu})$ is non-increasing, 
$I_{\gamma,1}({\mu})\leq 0$ for all $\mu > 0$.
When $\sigma>4/N$, we have $I_{\gamma,1}({\mu})=-\infty$ for every $\mu>0$.
   
\end{lem}

\begin{proof}
We start by showing that the functional $E_{\gamma,1}$ is bounded from below and coercive over $M_{\mu}$ when $0<\sigma<4/N$. Indeed, we infer from the Gagliardo-Nirenberg inequality \eqref{G-N-H2-ineq} that 
\begin{align*}
E_{\gamma,1}(u)&=\frac{\gamma}{2}\int_{\R^N}|\Delta u|^2\, dx+\frac{1}{2}\int_{\R^N}|\nabla u|^2\, dx-\frac{1}{2\sigma+2}\int_{\R^N}|u|^{2\sigma+2}\, dx\\
&\geq \frac{\gamma}{2}\int_{\R^N}|\Delta u|^2\, dx-\frac{B_N(\sigma)\mu^{1+\sigma-\frac{\sigma N}{4}}}{2\sigma+2}\left(\int_{\R^N}|\Delta u|^2\, dx\right)^{\frac{\sigma N}{4}},
\end{align*}
which proves the claim. Now, let $u\in M_{\mu}$ and consider $u_{\lambda}(x)=\lambda^{\frac{N}{2}}u(\lambda x)$ for $\lambda>0$ so that $u_{\lambda}\in M_{\mu}$. Then,
\begin{equation}\label{lambda-E}
I_{\gamma,1}({\mu})\leq E_{\gamma,1}(u_{\lambda})= \frac{\gamma\lambda^4}{2}\int_{\R^N}|\Delta u|^2\, dx+\frac{\lambda^2}{2}\int_{\R^N}|\nabla u|^2\, dx-\frac{\lambda^{\sigma N}}{2\sigma+2}\int_{\R^N}|u|^{2\sigma+2}\, dx
\end{equation}
for all $\lambda>0$. Letting $\lambda$ go to zero, we get $I_{\gamma,1}({\mu})\leq 0$ (we are not claimming that $I_{\gamma,1}({\mu})$ is negative). Note that the so-called \textit{large inequalities}
\begin{equation}\label{Large-ineq}
 I_{\gamma,1}({\mu})\leq I_{\gamma,1}(\theta)+I_{\gamma,1}(\mu-\theta),\hspace{.5cm}\mbox{for all}\ \theta\in ]0,\mu[,
\end{equation}
always hold true. 
Indeed, for any $\varepsilon>0$ we may choose test functions $u_{\varepsilon}\in M_{\theta}$ and $v_{\varepsilon}\in M_{\mu-\theta}$ with compact supports such that 
\begin{equation*}
I_{\gamma,1}(\theta)\leq E_{\gamma,1}(u_{\varepsilon})\leq I_{\gamma,1}(\theta)+\varepsilon,\hspace{.5cm}I_{\gamma,1}(\mu-\theta)\leq E_{\gamma,1}(v_{\varepsilon})\leq I_{\gamma,1}(\mu-\theta)+\varepsilon.
\end{equation*}
Then, if $e\in\R^N$ is a unit vector, we have that for $k$ large enough the supports of $u_{\varepsilon}$ and $v_{\varepsilon}(\cdot+ke)$ are disjoint. So, using the translation invariance of $E_{\gamma,1}$ and $M_{\mu}$ we have $u_{\varepsilon}+v({\cdot+ke})\in M_{\mu}$ for $k$ large and therefore
\begin{equation*}
 I_{\gamma,1}({\mu})\leq\limsup_{k\to\infty}E_{\gamma,1}(u_{\varepsilon}+v_{\varepsilon}(\cdot+ ke))\leq I_{\gamma,1}(\theta)+I_{\gamma,1}(\mu-\theta)+2\varepsilon.
\end{equation*}
Hence, \eqref{Large-ineq} holds and as a consequence we infer that $\mu\mapsto I_{\gamma,1}({\mu})$ is non-increasing since $I_{\gamma,1}(\mu)\leq 0$ for all $\mu$.
We finally observe that the last claim follows by letting $\lambda\to \infty$ in \eqref{lambda-E} when $\sigma>4/N$. 
\end{proof} 
In the next lemma we show that $I_{\gamma,1}({\mu})$ is strictly negative when $\sigma$ is $H^1$-subcritical or when $\gamma$ tends to $0$ or $\infty$. Those assertions are true without any restriction on $\mu$.  
\begin{lem}\label{Neg-Min}
For any given $\mu>0$, we have
 \begin{enumerate}[(a)]
  \item if $0<\sigma<2/N$, then $I_{\gamma,1}({\mu})<0$;
  \item if $2/N<\sigma<4/N$, there exists $\gamma_\mu>0$ such that $I_{\gamma,1}({\mu})<0$ for $\gamma<\gamma_\mu$.  
 \end{enumerate}
\end{lem}
\begin{proof}
 Let us begin by considering the case $0<\sigma<2/N$. Take $u\in M_{\mu}$ and set $u_{\lambda}(x)=\lambda^{\frac{N}{2}}u(\lambda x)$ for $\lambda>0$ so that $u_{\lambda}\in M_{\mu}$. Then observe that
 \begin{equation*}
 \frac{E_{\gamma,1}(u_{\lambda})}{\lambda^{\sigma N}}=\frac{\gamma\lambda^{4-\sigma N}}{2}\int_{\R^N}|\Delta u|^2\, dx+\frac{\lambda^{2-\sigma N}}{2}\int_{\R^N}|\nabla u|^2\, dx-\frac{1}{2\sigma+2}\int_{\R^N}|u|^{2\sigma+2}\, dx. 
 \end{equation*}
Assertion $(a)$ follows by taking $\lambda$ small enough. Next, we assume $2/N<\sigma<4/N$ and $\gamma\rightarrow 0 $. 
Assertion $(b)$ follows by taking $\lambda=1/\sqrt{\gamma}$ and noticing that $\gamma^{\frac{\sigma N-2}{2}} \rightarrow 0$ as $\gamma\to 0$.
\end{proof}

\begin{rmq}[Estimates for $\gamma$ fixed]\label{Neg-Minrmq}{\ }
\begin{enumerate}
\item One can deduce, arguing as in the previous lemma that if $2/N<\sigma<4/N$ and $\beta\rightarrow 0 $, then $I_{\gamma,\beta}({\mu})<0$; 
\item When $\beta =0$, it is easy to see that $I_{\gamma,0}(\mu)<0$ for any $0<\sigma < 4/N$ since
$$ \frac{E_{\gamma,0}(u_{\lambda})}{\lambda^{\sigma N}}=\frac{\gamma\lambda^{4-\sigma N}}{2}\int_{\R^N}|\Delta u|^2\, dx-\frac{1}{2\sigma+2}\int_{\R^N}|u|^{2\sigma+2}\, dx. $$
\end{enumerate}
\end{rmq}

\noindent Using the extended Gagliardo-Nirenberg interpolation inequality \eqref{GNtriple}, we now deduce some refined estimates on the sign of $I_{\gamma,1}({\mu})$.
\begin{lem}\label{LemmadeJB}
Let $2/N< \sigma<4/N$. Then, $E_{\gamma,1}$ takes negative values in $M_\mu$ if and only if the functional
$$u\mapsto\widetilde{C}^{\frac{\sigma N-4}{2}} \left(\int_{\R^N} |\nabla u|^2\, dx\right)^{\frac{4-\sigma N}{2}}\left(\int_{\R^N} |\Delta u|^2\, dx\right)^{\frac{\sigma N-2}{2}} -\int_{\R^N} |u|^{2\sigma +2}\, dx $$
takes negative values in $M_\mu$, where
$$\widetilde{C} = 2\gamma \left(\dfrac{1}{\sigma N-2}-\dfrac{1}{2}\right) \left( \dfrac{\gamma (2\sigma  +2)}{\sigma N-2}\right)^{\frac{2}{\sigma N -4}}.$$
We also have, for all $u\in M_\mu$,
$$E_{\gamma,1}(u)\geq \lambda (u) \int_{\R^N}|\nabla u|^2\, dx \left(\dfrac{1}{2} -V \mu^{\frac{2\sigma}{4-\sigma N}}\right),$$
where
\begin{equation}\label{lambda-eq-u}
\lambda(u)= \left( \dfrac{\gamma (2\sigma  +2)}{(\sigma N-2)} \dfrac{\int_{\R^N} |\Delta u|^2\, dx }{\int_{\R^N} |u|^{2\sigma +2}\, dx} \right)^{\frac{2}{\sigma N-4}} 
\end{equation}
and 
\begin{equation*}
V=\gamma \left(\dfrac{1}{\sigma N-2}-\dfrac{1}{2}\right) (C_{\sigma,N})^{\frac{2}{4-\sigma N}}\left(\dfrac{\gamma (2\sigma  +2)}{\sigma N-2}\right)^{\frac{2}{\sigma N-4}}. 
\end{equation*}
Moreover, setting $\mu_c=\left(\frac{1}{2V}\right)^{\frac{4-\sigma N}{2\sigma}}$, we have $ I_{\gamma,1}({\mu}) =0$ if and only if $\mu \leq \mu_c$. In case $\sigma=2/N$, we have $I_{\gamma,1}({\mu}) =0$ if and only if $\mu \leq \left(\dfrac{1}{2C_N (\frac{2}{N})}\right)^{\frac{N}{2}}$.
\end{lem}

\begin{proof}
Let $u\in M_\mu$ and consider $u_\lambda (x)=\lambda^{\frac{N}{4}}u(\sqrt{\lambda} x)$ for $\lambda>0$. Then
$$\dfrac{E_{\gamma,1} (u_{\lambda})}{\lambda}= \dfrac{\gamma}{2}\lambda \int_{\R^N} |\Delta u|^2\, dx+ \dfrac{1}{2} \int_{\R^N} |\nabla u|^2\, dx-\dfrac{\lambda^{\frac{\sigma N}{2} -1}}{2\sigma +2}\int_{\R^N} u^{2\sigma +2}\, dx. $$
The minimum of the right-hand side with respect to $\lambda$ is achieved at
$\lambda(u)$, which is given in \eqref{lambda-eq-u}, 
and its value is
$$
 \dfrac{1}{2} \int_{\R^N} |\nabla u|^2\, dx-\gamma \left(\dfrac{1}{\sigma N-2}-\dfrac{1}{2}\right) \left( \dfrac{\gamma (2\sigma  +2)}{\sigma N-2}\right)^{\frac{2}{\sigma N -4}} \left(\int_{\R^N} |\Delta u|^2\, dx\right)^{\frac{\sigma N-2}{\sigma N-4}} \left(\int_{\R^N} |u|^{2\sigma +2}\, dx\right)^{\frac{2}{4-\sigma N}},
$$
which proves the first claim of this lemma. Next, using \eqref{GNtriple}, we observe that
$$ \left(\int_{\R^N}|\Delta u|^2\, dx\right)^{\frac{\sigma N-2}{\sigma N -4}} \left(\int_{\R^N} |u|^{2\sigma +2}\, dx \right)^{\frac{2}{4-\sigma N }} \leq \left(C_{\sigma,N} \mu^{\sigma} \right)^{\frac{2}{4-\sigma N }}\int_{\R^N}|\nabla u|^2\, dx.$$
Combining the two previous lines, we get
$$E_{\gamma,1} (u_{\lambda}) \geq \lambda (u) \int_{\R^N}|\nabla u|^2 dx \left(\dfrac{1}{2} - V \mu^{\frac{2\sigma}{4-\sigma N}}\right).$$
Then the conclusion that $I_{\gamma,1}(\mu) =0$ for $\mu \leq \mu_c$ follows from Lemma \ref{Lem-Min}. The case $\sigma N=2$ can be treated in a similar way.
\end{proof}

\subsection{Subcritical $H^1$ exponents}
Thanks to Lemma \ref{Neg-Min}, we can establish the existence of a minimizer for  \eqref{MinL2fixed} when $\sigma$ is $H^1$-subcritical. Let us emphasize that the proof does only use the fact that $I_{\gamma,1}(\mu)$ is strictly negative.
\begin{proof}[Proof of Theorem \ref{Compact-Min-Sol} in the $H^1$-subcritical case]
Recall that we always have the following inequality (see the proof of Lemma \ref{Lem-Min})
\begin{equation*}
 I_{\gamma,1}(\mu)\leq I_{\gamma,1}(\theta)+I_{\gamma,1}(\mu-\theta),\hspace{.5cm}\mbox{for all}\ \theta\in ]0,\mu[.
\end{equation*}
It is standard that the Concentration-Compactness method \cite{Li84} yields that the minimizing sequences, up to translations, are relatively compact
if and only if 
the \textit{strict subaddivity condition} holds, namely
\begin{equation}\label{Strict-Subb}
I_{\gamma,1}({\mu})< I_{\gamma,1}(\theta)+I_{\gamma,1}(\mu-\theta),\hspace{.5cm}\mbox{for all}\ \theta\in ]0,\mu[. 
\end{equation}  
In fact, arguing as in \cite{Li84}, the inequality \eqref{Strict-Subb} is easily obtained provided $I_{\gamma,1}({\mu})<0$ which is the case when $0<\sigma<2/N$.
\end{proof}



\subsection{Subcritical $H^2$ exponents}
This case is a bit more involved at least when $\mu=\mu_c$, where $\mu_c$ stands for the critical mass. In this last case, we express all the integrals involved in $E_{\gamma,1}$ in terms of $I_{\gamma,1}({\mu})$ and the Lagrange multiplier associated to $I_{\gamma,1}({\mu})$, which we denote by $\alpha$.
\begin{lem}
\label{lastlemsec2}
Let $2/N<\sigma<4/N$. Then any minimizer $u$ of $ I_{\gamma,1}({\mu})$ satisfies
\begin{align*}
\int_{\R^N} |\Delta u|^2\, dx &=\left(\dfrac{\sigma N-2\sigma-2}{2\sigma\gamma}\right)2I_{\gamma,1}({\mu})+\left(\dfrac{\sigma N-2}{2\sigma\gamma}\right)\alpha\mu,\\
\int_{\R^N} |\nabla u|^2\, dx &=\left(\dfrac{4\sigma+4-\sigma N}{2\sigma}\right) 2I_{\gamma,1}({\mu}) +\alpha\mu\left(\dfrac{4-\sigma N}{2\sigma}\right),\\
\int_{\R^N}|u|^{2\sigma +2}\, dx &= \left(\frac{\sigma+1}{\sigma}\right)\left(2 I_{\gamma,1}({\mu})+\alpha\mu\right),
\end{align*}
where
$$\alpha= \dfrac{2\sigma \gamma \int_{\R^N} |\Delta u|^2 dx -(\sigma N-2\sigma-2) 2I_{\gamma,1}({\mu})}{\mu (\sigma N-2)}.$$
Let $\sigma=2/N$ and assume $ I_{\gamma,1}({\mu})=0$. Then, we have
$$\int_{\R^N} |u|^{2\sigma +2}\, dx= (\sigma +1) \int_{\R^N}|\nabla u|^2\, dx = \dfrac{\sigma+1}{\sigma} \alpha\mu.$$
\end{lem}

\begin{proof}
Let $u$ be any minimizer of $I_{\gamma,1}({\mu})$ and let $\alpha$ be the associated Lagrange multiplier. By definition, we have
$$\gamma \int_{\R^N} |\Delta u|^2\, dx +\int_{\R^N} |\nabla u|^2\, dx - \int_{\R^N}|u|^{2\sigma +2}\, dx +\alpha \mu=0,$$
and
$$E_{\gamma,1}(u)=\dfrac{\gamma}{2}  \int_{\R^N} |\Delta u|^2\, dx +\dfrac{1}{2}\int_{\R^N} |\nabla u|^2\, dx - \dfrac{1}{2\sigma +2}\int_{\R^N}|u|^{2\sigma +2}\, dx = I_{\gamma,1}({\mu}).$$
On the one hand, by using the Derrick-Pohozaev identity we obtain
$$2\gamma  \int_{\R^N} |\Delta u|^2\, dx +\int_{\R^N} |\nabla u|^2\, dx - \dfrac{\sigma N}{2\sigma +2}\int_{\R^N}|u|^{2\sigma +2}\, dx=0.$$
As a consequence of the two previous equalities, we have
$$\int_{\R^N}|u|^{2\sigma +2}\, dx = \dfrac{2(2\sigma +2)}{\sigma N-2} \left( I_{\gamma,1}({\mu}) + \dfrac{\gamma}{2} \int_{\R^N} |\Delta u|^2\, dx \right).$$
On the other hand, using again the Derrick-Pohozaev identity, we get
\begin{align*}
\alpha \mu &=\gamma  \int_{\R^N} |\Delta u|^2\, dx - \left(\dfrac{\sigma N -2\sigma -2}{2\sigma +2}\right) \int_{\R^N}|u|^{2\sigma +2}\, dx.
\end{align*}
Thus, we deduce from the previous lines that
$$\alpha= \dfrac{2\sigma \gamma \int_{\R^N} |\Delta u|^2\, dx -(\sigma N-2\sigma-2) 2I_{\gamma,1}({\mu})}{\mu (\sigma N-2)}.$$
Therefore, the proof of the lemma follows from straightforward computations.
\end{proof}
Now, we are finally in position to finish the proof of the existence part of Theorem \ref{Compact-Min-Sol}. The next proposition states the assertions to be proved.

\begin{prop}
Let $2/N< \sigma<4/N$. Then, the following assertions hold true :
\begin{enumerate}[(i)]
\item If $\mu< \mu_c$, then $ I_{\gamma,1}({\mu})$ is not achieved.
\item If $\mu \geq \mu_c$, then $ I_{\gamma,1}({\mu})$ is achieved.
\end{enumerate}
Let $\sigma = 2/N$. Then, we have
\begin{enumerate}[(1)]
\item If $\mu < \left(\dfrac{1}{2C_N (\frac{2}{N})}\right)^{\frac{N}{2}}$, then $ I_{\gamma,1}({\mu})$ is not achieved.
\item If $\mu > \left(\dfrac{1}{2C_N (\frac{2}{N})}\right)^{\frac{N}{2}}$, then $ I_{\gamma,1}({\mu})$ is achieved.
\end{enumerate}
\end{prop}

\begin{proof}
$(i)$ Seeking a contradiction, let us suppose that there exists $\widetilde{\mu}\in (0,\mu_c)$ such that $I_{\gamma,1}(\widetilde{\mu})$ has a minimizer $u_{\widetilde{\mu}}$. From the definition of $\mu_c$, we have that $I_{\gamma,1}(\mu_c)=0$. It is easy to check that, for $t>1$, we have
\begin{equation*}
 I_{\gamma,1}(t\widetilde{\mu})\leq E_{\gamma,1}(\sqrt{t}u_{\widetilde{\mu}})<tE_{\gamma,1}(u_{\widetilde{\mu}})=tI_{\gamma,1}(\widetilde{\mu}),
\end{equation*}
which implies that $I_{\gamma,1}(\mu_L)<0$ for any $\mu_L>\widetilde{\mu}$. Hence, we deduced a contradiction with the definition of $\mu_c$. \medbreak

\noindent $(ii)$ Set $\mu_k =\mu_c +1/k$ for every $k\in\N$. One can show that $I_{\gamma,1}(\mu_k)\rightarrow I_{\gamma,1}(\mu_c)=0$ as $k\to\infty$. Moreover, from Lemma \ref{Neg-Min} we know that $I_{\gamma,1}(\mu_k)<0$ and then it admits a minimizer $u_k$ bounded in $H^2$. Now, we claim that 
\begin{equation*}
\int_{\R^N}|u_k|^{2\sigma +2}\, dx\geq \delta, 
\end{equation*}
for some constant $\delta$ independent of $k$. Indeed, by using the Gagliardo-Nirenberg inequality \eqref{GNtriplegen}, we infer that
\begin{equation}
\label{GNtriplelast}
\|u \|_{L^{2\sigma+ 2}}^{2\sigma +2} \leq C \mu^{\beta_1} \|\nabla u \|_{L^2}^{\beta_2}\|\Delta u \|_{L^2}^{\beta_3},
\end{equation}
where $\beta_1,\beta_2,\beta_3$ are appropriate constants such that $\beta_2+\beta_3> 2$. Let 
$$\alpha_k= \dfrac{2\sigma \gamma \int_{\R^N} |\Delta u_k|^2 dx-(\sigma N-2-2\sigma) 2I_{\gamma,1}({\mu_k})}{\mu_k (\sigma N-2)}.$$
Now, combining Lemma \ref{lastlemsec2} and \eqref{GNtriplelast}, and since $I_{\gamma,1}(\mu_k)\rightarrow 0$, we obtain
$$\dfrac{\sigma+1}{\sigma}\leq C\liminf_{k\rightarrow \infty} \mu_k^{\beta_1}(\alpha_k \mu_k)^{[(\beta_2+\beta_3)/2]-1}.$$
This gives that $ \liminf_{k\rightarrow \infty} (\alpha_k\mu_k)>0$, which in view of Lemma \ref{lastlemsec2} proves the claim.
 Thus, we have excluded the vanishing. Then proceeding along the lines of \cite[Proposition 3.4]{dolb}, one shows that, up to a subsequence, $u_k$ converges weakly in $H^2 (\R^N)$ and strongly in $L^2_{\loc}(\R^N)$ to a function $u_\infty$ satisfying $E_{\gamma,1}(u_\infty)=0$ and $\|u_\infty\|_{L^2 (\R^N)}=\mu_c$. 
 
\medbreak

\noindent The assertion (1) follows with the arguments used for the proof of (i) whereas (2) was already proved in Lemma \ref{LemmadeJB}.

\end{proof}

\section{Qualitative properties}\label{sec:quali}
In this section we investigate different qualitatives properties of solutions to \eqref{MinLpintro} and \eqref{MinL2fixed}. First, we study the positivity of standing waves with a prescribed mass. If we substitute the term $|u|^{2\sigma +2}$ by $|u^+|^{2\sigma +2}$ in $E_{1,\beta}$, where $u^{+}(x)$ stands for $\max\{u(x),0\}$, we prove the positivity of solutions to \eqref{MinL2fixed} provided that the Lagrange multiplier is small enough. We show that this is the case if the mass is sufficiently small. On the other hand, if the Lagrange multiplier is too large, we prove that radial solutions to \eqref{MinL2fixed} are sign-changing. Afterwards we investigate the symmetry properties of solutions to \eqref{MinLpintro} and \eqref{MinL2fixed}. We prove that \eqref{MinL2fixed} has at least one solution that is radially symmetric whenever $2\sigma \in \N$. Concerning solutions to \eqref{MinLpintro}, their radial symmetry was already established in \cite{bonnas} using the results of \cite{MR1755067} which are based on the moving-plane procedure. We give an alternative proof of this fact based on rearrangement methods, namely Talenti's principle and P\'olya-Szeg\"o's inequality. Then, we establish the exponential decay of any solution to \eqref{eq4beta} going to $0$ as $|x|\rightarrow \infty$. Finally, we show the nondegeneracy of the unique minimal solution when the dimension is one.
\subsection{Existence of positive standing waves with a prescribed mass}
In this section we fix $\gamma=1$ and we consider the minimization problem
\begin{equation}\label{MinL2fixed1}
 \widetilde{I}_{1,\beta}({\mu})=\inf_{u\in M_{\mu}}\widetilde{E}_{1,\beta}(u)
\end{equation}
where $M_\mu$ is defined as in \eqref{de:Mmu} and
\begin{equation}\label{def:Egammabeta1}
\widetilde{E}_{1,\beta}(u)=\frac{1}{2}\int_{\R^N}|\Delta u|^2\, dx+\frac{\beta}{2}\int_{\R^N}|\nabla u|^2\, dx-\frac{1}{2\sigma+2}\int_{\R^N}|u^+|^{2\sigma+2}\, dx.
\end{equation}
The proof of Theorem \ref{Compact-Min-Sol} applies to problem \eqref{MinL2fixed1} with straightforward modifications. Let us recall that if $u$ is a solution to problem \eqref{MinL2fixed1}, then $u$ solves
\begin{equation}\label{alpha-mu-eq}
 \Delta^2 u -\beta \Delta u +\alpha (\mu) u= |u^+|^{2\sigma}u^+,
\end{equation}
where
\begin{equation}\label{alpha-mu}
-\alpha (\mu)\mu=\int_{\R^N} ( |\Delta u|^2+\beta |\nabla u|^2 - |u|^{2\sigma +2})\,  dx=2\widetilde{E}_{1,\beta}(u) - \dfrac{\sigma}{\sigma+1}\int_{\R^N}  |u^+|^{2\sigma +2}\, dx.
\end{equation}
It is immediate to see that $\alpha (\mu)\geq 0$ since $\widetilde{E}_{1,\beta}(u)<0$. We establish the positivity of solutions to \eqref{MinL2fixed1} provided that $\alpha (\mu)$ is small enough.
\begin{thm}
\label{postand}
Suppose that $ \beta \geq 2 \sqrt{\alpha(\mu)}$. Then, any solution to \eqref{MinL2fixed1} is strictly positive (or strictly negative).
\end{thm}
\begin{proof}
Let $u$ be a solution to \eqref{MinL2fixed1}. When the Lagrange multiplier is smaller than $\beta^2/4$, we can rewrite the equation satisfied by $u$ as
$$\left\{\begin{array}{ll}
-\Delta u+\lambda_1 u&=v, \\
-\Delta v+\lambda_2 v&=|u^+|^{2\sigma +1},
\end{array}\right. $$
for some positive constants $\lambda_1,\lambda_2$, which satisfy $\lambda_1 \lambda_2=\alpha (\mu)$ and $\lambda_1 +\lambda_2 = \beta$. It is then standard to see that $u > 0$.
\end{proof}

We next estimate the Lagrange multiplier of problem \eqref{MinL2fixed1} by the $L^2$-mass, namely we prove 
$$\alpha(\mu)\leq C\mu^{\frac{\sigma}{1-\sigma N/4}}$$
for some $C>0$. This estimate enables us to apply the previous theorem when the mass is small enough. We first recall the Derrick-Pohozaev identity associated to \eqref{MinL2fixed1}. We give a proof for completeness. 
\begin{lem}
Let $u$ be a solution to \eqref{MinL2fixed1}. Then 
\begin{equation}\label{Pohozaev}
 2\int_{\R^N}|\Delta u|^2\, dx+\beta\int_{\R^N}|\nabla u|^2\, dx=\frac{\sigma N}{2\sigma+2}\int_{\R^N}|u^+|^{2\sigma+2}\, dx.
\end{equation}
\end{lem}

\begin{proof}
Set $u_{\lambda}(x)$=$\lambda^{\frac{N}{2}}u(\lambda x)$ for $\lambda>0$ so that $\|u_{\lambda}\|_{L^2}=\|u\|_{L^2}$. Note that 
\begin{equation*}
 \widetilde{E}_{1,\beta}(u_{\lambda})=\frac{\lambda^4}{2}\int_{\R^N}|\Delta u|^2\, dx+\frac{\beta\lambda^2}{2}\int_{\R^N}|\nabla u|^2\, dx-\frac{\lambda^{\sigma N}}{2\sigma+2}\int_{\R^N}|u^+|^{2\sigma+2}\, dx.
\end{equation*}
By differentiating the previous line with respect to $\lambda$ and using the minimality of $u$, we get
\begin{equation*}
 0=\frac{d \widetilde{E}_{1,\beta}(u_{\lambda})}{d\lambda}\Bigg{|}_{\lambda=1}=2\int_{\R^N}|\Delta u|^2\, dx+\beta\int_{\R^N}|\nabla u|^2\, dx-\frac{\sigma N}{2\sigma+2}\int_{\R^N}|u^+|^{2\sigma+2}\, dx.
\end{equation*}
\end{proof}

From \eqref{alpha-mu} and \eqref{Pohozaev}, we deduce that
\begin{equation}\label{expressalpha}
\alpha(\mu)=\frac{1}{\mu}\left(\int_{\R^N}|\Delta u|^2\, dx+\left(1-\frac{\sigma N}{2\sigma+2}\right)\int_{\R^N}|u^+|^{2\sigma+2}\, dx\right).
\end{equation}

This expression allows us to estimate the Lagrange multiplier for small values of $\mu$.
\begin{lem}
Assume $0< \sigma<4/N$ and $\beta>0$. 
Then,  
$$\alpha(\mu)\leq B_N(\sigma)^{\frac{1}{1-\sigma N/4}}\left(2-\frac{\sigma N}{2\sigma+2}\right)\mu^{\frac{\sigma}{1-\sigma N/4}}.$$
\end{lem}
\begin{proof}
From \eqref{alpha-mu-eq}, we deduce that
$$\int_{\R^N}|\Delta u|^2\, dx\le \int_{\R^N}|u^+|^{2\sigma+2}\, dx,$$
which, from \eqref{expressalpha}, yields
\begin{equation*} 
 0\leq\alpha(\mu)\leq\frac{1}{\mu}\left(2-\frac{\sigma N}{2\sigma+2}\right)\int_{\R^N}|u^+|^{2\sigma+2}\, dx.
\end{equation*}
The Gagliardo-Nirenberg inequality \eqref{G-N-H2-ineq} then implies that
\begin{align*}
\int_{\R^N}|u^+|^{2\sigma+2}\, dx & \leq  \int_{\R^N}|u|^{2\sigma+2}\, dx\\
& \leq  B_N(\sigma)\left(\int_{\R^N}|\Delta u|^2\, dx\right)^{\frac{\sigma N}{4}}\left(\int_{\R^N}|u|^{2}\, dx\right)^{1+\sigma-\frac{\sigma N}{4}}\\
& \leq B_N(\sigma)\left(\int_{\R^N}|u^+|^{2\sigma+2}\, dx\right)^{\frac{\sigma N}{4}}\left(\int_{\R^N}|u|^{2}\, dx\right)^{1+\sigma-\frac{\sigma N}{4}}.
\end{align*}
Hence, we conclude that
\begin{equation*}
\int_{\R^N}|u^+|^{2\sigma+2}\, dx\leq B_N(\sigma)^{\frac{1}{1-\sigma N/4}}\mu^{1+\frac{\sigma}{1-\sigma N/4}}, 
\end{equation*}
which yields to
\begin{equation*}
 \alpha(\mu)\leq B_N(\sigma)^{\frac{1}{1-\sigma N/4}}\left(2-\frac{\sigma N}{2\sigma+2}\right)\mu^{\frac{\sigma}{1-\sigma N/4}}.
\end{equation*}
\end{proof}
\begin{cor}
Assume $0< \sigma<4/N$ and $\beta>0$.  
Then, there exists $\mu_0>0$ such that 
$$2 \sqrt{\alpha(\mu)}\le \beta,\hspace{.2cm}\text{for all}\ \mu\leq \mu_0$$
and therefore any solution to \eqref{MinL2fixed1} is strictly positive (or strictly negative). 
\end{cor}

\begin{rmq}
\label{lagrunifest}
We point out that it is possible to show that $\alpha (\mu)<\beta^2/4$ for a given mass $\mu$ under convenient assumptions on the coefficients. Indeed, let us consider more generally the case $\gamma \neq 1$. In this situation, we want to show that $\alpha (\mu)<\beta^2/(4\gamma)$. We have
$$-\alpha (\mu)\mu=\int_{\R^N} (\gamma |\Delta u|^2+\beta |\nabla u|^2 - |u^+|^{2\sigma +2})\, dx,$$
and the Derrick-Pohozaev identity rewrites as
$$2\gamma\int_{\R^N}|\Delta u|^2\, dx+\beta\int_{\R^N}|\nabla u|^2\, dx=\frac{\sigma N}{2\sigma+2}\int_{\R^N}|u^+|^{2\sigma+2}\, dx.$$
One can show, as previously, that
$$\alpha(\mu)\leq\frac{C}{\mu}\int_{\R^N}|u^+|^{2\sigma+2}\, dx,$$
for some constant $C$ independent of $\gamma$ and $\beta$. On the other hand, using \eqref{G-N-H2-ineq} we have
$$\int_{\R^N}|u^+|^{2\sigma+2}\, dx\leq B_{N}(\sigma) \gamma^{-\frac{\sigma N}{4-\sigma N}} \mu^{1+\frac{\sigma}{1-\sigma N/4}}.$$
Combining the two previous inequalities, we deduce that
$\gamma \alpha(\mu)\leq C \gamma^{2\frac{2-\sigma N}{4-\sigma N}} \mu^{\frac{\sigma}{1-\sigma N/4}}$. Thus, taking $\gamma$  sufficiently large if $2/N<\sigma<4/N$ or sufficiently small if $\sigma < 2/N$, we deduce that $\gamma\alpha (\mu)<\beta^2/4$. If $\sigma < 2/N$ and $\gamma=1$, using \eqref{G-N-H1-ineq} instead of \eqref{G-N-H2-ineq}, we get that
$ \alpha(\mu)\leq C \beta^{-C_2} \mu^{\frac{\sigma}{1-\sigma N/4}}$, for some constant $C_2>0$. Thus we have $\alpha (\mu)<\beta^2/4$ provided that $\beta$ is large enough.
\end{rmq}
It follows from the previous remark that if $0<\sigma < 2/N$, we have $\alpha (\mu)<\beta^2/(4\gamma)$ asymptotically for $\beta\to +\infty$ or $\gamma\to 0$. In these cases, we therefore deduce the sign of any solution to \eqref{MinL2fixed1}. This is expected as the sign is known for the second order limit problem. We investigate these asymptotic regimes in Section \ref{sec:smalldis}.

\subsection{Radial symmetry of at least one minimal standing wave with prescribed mass}
Using the method of \cite{boulengerlenzmann}, one can show that at least one solution to \eqref{MinL2fixed} is radially symmetric if $2\sigma \in \N$. This proposition completes the proof of Theorem \ref{Compact-Min-Sol}.
\begin{prop}
Suppose that problem \eqref{MinL2fixed} has a minimizer and assume $2\sigma \in \N$. Then there exists at least one radially symmetric minimizer for \eqref{MinL2fixed}.
\end{prop}
\begin{proof}
The proof is a direct adaptation from \cite[Appendix A.2]{boulengerlenzmann}. The main ingredient of the proof is the \textit{Fourier rearrangement}. Namely, for any $u\in L^2(\R^N)$ we set its Fourier rearrangement by $u^\sharp= \mathcal{F}^{-1} \{ (\mathcal{F} u )^\ast \}$, where $\mathcal{F}$ stands for the Fourier transform and $f^\ast$ denotes the symmetric-decreasing rearrangement of a measurable function $f:\R^N\to\R$ that vanishes at infinity. Observe that $\left\| u^\sharp \right\|_{L^2}= \left\| u \right\|_{L^2}$. Then, assuming that $u\in H^2(\R^N)$, we have \cite{boulengerlenzmann}
\begin{equation}
\label{eqcaseboulenz}
\left\|\Delta u^\sharp \right\|_{L^2} \leq \left\|\Delta u \right\|_{L^2},
\end{equation}
$$\left\|\nabla u^\sharp \right\|_{L^2} \leq \left\|\nabla u \right\|_{L^2},$$
and
$$\left\| u \right\|_{L^{2m}}\leq \left\| u^\sharp \right\|_{L^{2m}},$$
for any $m\in \N$. Therefore, if $u$ is a minimizer for \eqref{MinL2fixed}, then $u^\sharp$ is a minimizer as well.
\end{proof}
It is an open problem to extend the previous proposition for $2\sigma\not\in\N$. Observe also that we do not know whether or not all solutions to \eqref{MinL2fixed} are radially symmetric even if $2\sigma \in \N$. Indeed, Boulenger and Lenzmann proved that equality holds in \eqref{eqcaseboulenz} if and only if $|\mathcal{F} u|=|\mathcal{F}u|^\ast$ and it is not clear that this implies that $u$ is a radial function.

\subsection{Radial sign-changing minimal standing waves with prescribed mass}
In this subsection we show that the restriction on the Lagrange multiplier of Theorem \ref{postand} is sharp (in some sense) to obtain the positivity of minimizers. Indeed, we have the following theorem.
\begin{thm}
\label{thmradialsignchanging}
Let $v\in H^2(\R^N)$ be any radial solution to \eqref{eq4beta} with $0<\sigma<2/(N-4)$ if $N>4$ ($\sigma>0$ if $N\le 4$) and $-2\sqrt{\alpha}<\beta < 2\sqrt{\alpha}$. Then $v$ is sign-changing.
\end{thm}
\begin{proof}
We argue as in \cite[Theorem 4]{bonnas} where the cases $N=2,3$ and $\sigma=3$ were considered. We first claim that $v\in W^{4,q}_{\loc}(\R^N)$ for every $q\ge 1$ with a uniform estimate on unit cubes. In particular, $v\in W^{3,\infty}(\R^N)$. 
Writing the equation as 
$$-\Delta (-\Delta v) +\beta (-\Delta v)=|v|^{2\sigma}v-\alpha v\ \text{in}\ \R^N,$$
and observing that $|v|^{2\sigma}v-\alpha v\in L^p_{\loc}(\R^N)$ for $1\le p\le \frac{2N}{(N-4)(2\sigma+1)}$ if $N>4$ (every $p\ge 1$ if $N\le 4$), we deduce from local elliptic regularity estimates, see e.g. \cite[Chapter $5$, Theorem $5$]{kry}, that $\Delta v\in W^{2,p}_{\loc}(\R^N)$ which in turn implies $v\in W^{4,p}_{\loc}(\R^N)$. This local $W^{4,p}$-bound is uniform on unit cubes. Using Sobolev imbedding and bootstrapping the previous argument, we infer that $v\in W^{4,q}_{\loc}(\R^N)$ for every $q\ge 1$. Given $q>N$, we obtain a uniform estimate of the $W^{4,q}$-bound on unit cubes. This implies $v\in W^{3,\infty}(\R^N)$.

Now since $v\in W^{1,\infty}(\R^N)\cap H^2(\R^N)$, we have $|v|^{2\sigma}v-\alpha v\in W^{1,p}(\R^N)$ for every $p\ge 2$. By elliptic regularity, see e.g. \cite[Chapter $5$, Theorem $1$]{kry}, this implies $v\in W^{5,p}(\R^N)$ for every $p\ge 2$ and therefore $v\in C^{4,\alpha}$ for every $0<\alpha<1$. This shows on the one hand that $v$ is a classical solution to \eqref{eq4beta} and on the other hand that $\partial^\tau v$ vanishes at infinity for every multi-index $\tau$ of length smaller than $3$. 

Next, arguing as in de Figueiredo et al \cite[Theorem 2.2]{deFig-MS-Mi}, we can extend \cite[Lemma 4]{bonnas} to $W^{m,p}_{rad}(\R^N)$ for every $p\ge 1$, $m\ge 1$ and $N\ge 1$ and therefore we can complete the argument by proceeding exactly as in the proof of \cite[Theorem 4]{bonnas}. 
\end{proof}
As a direct corollary of the previous theorem, we are able to prove that solutions obtained in Theorem \ref{Compact-Min-Sol} (see Remark \ref{Neg-Minrmq}) when $\beta=0$ are sign-changing provided they are radial.
\begin{cor}
\label{sgnchmassfix}
Assume that $0<\sigma < 4/N$ and $2\sigma \in \N$. Then there exists a sign-changing solution to \eqref{MinL2fixed}.
\end{cor}

\subsection{Radial symmetry of all ground states in the strong second order dispersion case or weak fourth order dispersion case}
In \cite{bonnas}, the first and fourth authors have dealt with the minimization problem
\begin{equation}
\label{min1}
m=\inf_{u\in M}J_{1,\beta,\alpha}(u),
\end{equation}
where $J_{1,\beta,\alpha}(u)$ and $M$ are defined in \eqref{J-quadr} and \eqref{lpconsint} respectively. 
In this section we consider the same minimization problem on the set of complex-valued functions, namely we minimize the functional $J_{1,\beta,\alpha}$ over 
\begin{equation}
\label{lpconsintcomplex}
\overline M= \{u\in H^2 (\R^N;\mathbb{C}):\int_{\R^N} |u|^{2\sigma +2}\, dx=1\},
\end{equation}
The aim of this section is to prove that if $\beta>2\sqrt{\alpha}$, then any minimizer has the form $e^{i\tau}U$ for some $\tau\in \R$ where $U$ is radially symmetric around some point. In fact, it is already known from \cite[Theorem 1.1]{bonnas} that real minimizer are radially symmetric. However, here we provide a different proof that only requires the Riesz-Fr\'echet representation theorem and rearrangement techniques, namely Talenti's principle, the Hardy-Littlewood inequality and the P\'olya-Szeg\"o inequality. We believe this argument is in a way more direct than the moving plane techniques in its variant from \cite{MR1755067}. Also the proof given here is based on a purely variational argument which has its own interest.  

Given a nonnegative measurable function $f$ defined on $\R^N$ vanishing at infinity, we denote by $f^*$ the Schwarz symmetric function associated to $f$, that is, the radially symmetric, radially non-increasing function equi-measurable with $f$.

\begin{thm}\label{Radial-LES}
Assume that $0<\sigma<4/(N-4)$ if $N\geq 5$, $\alpha>0$ and $\beta>2\sqrt{\alpha}$. Then any solution $u$ to 
\begin{equation}
\label{min1complex}
\overline{m}=\inf_{u\in \overline M}J_{1,\beta,\alpha}(u)
\end{equation}
has the form $e^{i\tau}U$ where $U$ does not change sign. Moreover, assuming (without loss of generality) $U>0$ and $U(0)=\max_{\R^N}U$, we have $U=U^*$ and $\partial_r U(x)<0$ for every $x\ne 0$. 
\end{thm}

\begin{proof}
Let $u=u_1+iu_2$ be a solution to \eqref{min1complex}. 

\medbreak

\noindent\emph{Step one: $u = e^{i\tau}U$ for some $\tau\in \R$}. We start by showing that $|u|>0$. The argument can be compared with \cite[Lemma 4.1]{bonnas}. Define $w_j\in H^2 (\R^N)$, $j=1,2$, such that
$$-\Delta w_j +\dfrac{\beta}{2}w_j= \left|-\Delta u_j+ \dfrac{\beta}{2}u_j \right|\hspace{.3cm} \text{in}\ \R^N.$$
Thus, we have
$$-\Delta (w_j \pm u_j)+\dfrac{\beta}{2} (w_j\pm u_j)\geq 0.$$
By the strong maximum principle, we know that $u_j$ has a fixed sign if $-\Delta u_j + \beta u_j/2$ does not change sign. Seeking a contradiction, suppose that $-\Delta u_j + \beta u_j/2$ changes sign. The strong maximum principle shows that $w_j> |u_j|$. Therefore, for $w=w_1+iw_2$, we have
\begin{align*}
J_{1,\beta,\alpha}\left(\dfrac{w}{\left\|w\right\|_{L^{2\sigma+2}}}\right)&= \dfrac{\int_{\R^N } |-\Delta w+\beta w/2|^2\, dx- (\beta^2/4-\alpha) \int_{\R^N} |w|^2\, dx}{\|w\|^2_{L^{2\sigma+2}}}\\
&=\sum_{j=1}^{2}\dfrac{\int_{\R^N } (-\Delta w_j+\beta w_j/2)^2\, dx- (\beta^2/4-\alpha) \int_{\R^N} w_j^2\, dx}{\|w\|_{L^{2\sigma+2}}^2}\\
&< \sum_{j=1}^{2}\dfrac{\int_{\R^N } (-\Delta u_j+\beta u_j/2)^2\, dx- (\beta^2/4-\alpha) \int_{\R^N} u_j^2\, dx}{\|u\|_{L^{2\sigma+2}}^2}\\
&=\dfrac{\int_{\R^N } |-\Delta u+\beta u/2|^2\, dx - (\beta^2/4-\alpha) \int_{\R^N} |u|^2\, dx}{\|u\|^2_{L^{2\sigma+2}}}.
\end{align*}
This contradicts the minimality of $u$. Therefore, we conclude that $u_j$, $j=1,2$, do not change sign. 
Now that $|u|^{-1}\in L_{\loc}^\infty(\R^N)$, we infer that $D \text{sign}(u)=0$ where 
\begin{align*}
\text{sign}(z)=\left.
\begin{cases}
\dfrac{z}{|z|},& z\neq 0\\ 
0,& z=0.
\end{cases}
\right.
\end{align*}
Hence, there exists $\tau\in \R$ such that $\text{sgn}(u)=e^{i\tau}$, so that $u=e^{i\tau}|u|$. 

\medbreak

We can now assume that $U$ is positive, $-\Delta U + \frac\beta 2 U>0$ and by translation invariance, we fix $U(0)=\max_{\R^N}U$. 

\medbreak

\noindent\emph{Step two: $U=U^*$ and $\partial_r U(x)<0$ for every $x\ne 0$}.
Since $\beta>2\sqrt{\alpha}$, we may choose $\lambda\in\R$ such that
\begin{equation*}
 0<\beta<\lambda<\beta+\sqrt{\beta^2-4\alpha}.
\end{equation*}
For any $u_0\in M$,  we rewrite the quadratic form $J_{1,\beta,\alpha}(u_0)$ defined in \eqref{J-quadr} as
\begin{equation*}
 J_{1,\beta,\alpha}(u_0)=\int_{\R^N}\Big\rvert-\Delta u_0+\frac{\lambda}{2}u_0\Big\rvert^2\, dx+(\beta-\lambda)\left[\int_{\R^N}(|\nabla u_0|^2+\frac{\lambda}{2}|u_0|^2)\, dx\right]+
 \end{equation*}
\begin{equation*} 
 +P_{\alpha,\beta}\int_{\R^N}|u_0|^{2}\, dx
\end{equation*}
where $P_{\alpha,\beta}=\alpha-\frac{\beta\lambda}{2}+\frac{\lambda^2}{4}<0$.
Let $z\in H^1(\R^N)$ be such that
\begin{equation*}
-\Delta z+\frac{\lambda}{2}z=\left(-\Delta U+\frac{\lambda}{2}U\right)^*\hspace{.3cm}\mbox{in}\ \R^N.
\end{equation*}
Since $U\in H^2(\R^N)$, we have $f:=-\Delta U+\frac{\lambda}{2}U\in L^2(\R^N)$ and $f^*\in L^2(\R^N)$. Then 
 \begin{equation*}
\langle U,\phi\rangle_{\lambda}=\int_{\R^N}f\phi\, dx,\hspace{.3cm}\mbox{for all}\ \phi\in H^1(\R^N)
 \end{equation*}
and  
\begin{equation*}
\langle z,\phi\rangle_{\lambda}=\int_{\R^N}f^*\phi\, dx,\hspace{.3cm}\mbox{for all}\ \phi\in H^1(\R^N),
 \end{equation*}
where in the space $H^1(\R^N)$ we used the inner product defined through
 \begin{equation*}
  \langle \varphi,\psi\rangle_{\lambda}=\int_{\R^N}\nabla\varphi\nabla\psi\, dx+\frac{\lambda}{2}\int_{\R^N}\varphi\psi\, dx,\hspace{.3cm} \mbox{for all}\ \varphi,\psi\in H^1(\R^N),
\end{equation*}
with the corresponding norm
 \begin{equation*}
\|\varphi\|_{\lambda}=\left(\int_{\R^N}|\nabla \varphi|^2\, dx+\frac{\lambda}{2}\int_{\R^N}|\varphi|^2\, dx\right)^{1/2}.
 \end{equation*}
The Hardy-Littlewood inequality and the Cauchy-Schwarz inequality now give
\begin{equation}\label{HL-CL}
\|U\|^2_{\lambda}=\int_{\R^N}fU\, dx \leq \int_{\R^N}f^{*} U^{*}\, dx=\langle z,U^{*}\rangle_{\lambda} \leq\|z\|_{\lambda}\|U^{*}\|_{\lambda}.
\end{equation}
Hence, by applying the P\'olya-Szeg\"o inequality we infer that
\begin{equation}\label{P-S}
 \|U\|^2_{\lambda}\leq\|z\|_{\lambda}\|U^*\|_{\lambda}\leq\|z\|_{\lambda}\|U\|_{\lambda},
\end{equation}
which implies
\begin{equation}\label{encoreune}
 \|U\|^2_\lambda\leq\|z\|^2_{\lambda}.
\end{equation}
Moreover, it comes from \cite[Lemma 3.4]{BoMSRa}, see also inequality (9) in \cite{alvino1999}, that
\begin{equation}\label{BMR-ineq}
 \int_{\R^N}|U|^{2\sigma+2}\, dx\leq\int_{\R^N}|z|^{2\sigma+2}\, dx.
\end{equation}
We now claim that neither the inequality in \eqref{BMR-ineq} nor those in \eqref{HL-CL} and \eqref{P-S} is strict. Indeed, if this was the case, then it would imply
that
\begin{align*}
 J_{1,\beta,\alpha}\left(\frac{z}{\|z\|_{L^{2\sigma+2}}}\right)&=\displaystyle\frac{\int_{\R^N}(-\Delta z+\lambda z/2)^2\, dx+(\beta-\lambda)\|z\|^2_{\lambda}+P_{\alpha,\beta}\int_{\R^N}|z|^2\, dx}{\|z\|^2_{L^{2\sigma+2}}}\\
 &<\displaystyle\frac{\int_{\R^N}(-\Delta U+\lambda U/2)^2\, dx+(\beta-\lambda)\|U\|^2_{\lambda}+P_{\alpha,\beta}\int_{\R^N}|u|^2\, dx}{\|U\|^2_{L^{2\sigma+2}}},
\end{align*}
which contradicts the minimality of $U$. Therefore the equality holds in \eqref{HL-CL}, \eqref{P-S} and \eqref{BMR-ineq}. This implies $z=U^*$ and
\begin{equation}\label{grad-eq}
\int_{\R^N}|\nabla U^*|^2\, dx=\int_{\R^N}|\nabla U|^2\,dx.
\end{equation}
Moreover, $U^*$ is also a solution to \eqref{min1} and $U^*$ is smooth. 
Since 
\begin{equation*}
-\Delta U^*+\frac{\lambda}{2} U^*=\left(-\Delta U+\frac{\lambda}{2}U\right)^*\hspace{.3cm}\mbox{in}\ \R^N,
\end{equation*}
taking the derivative with respect to $r$ in the above equation, we get
\begin{equation*}
 -\Delta G+A(x)G\le 0,
\end{equation*}
where $G=\partial_r U^*$ and $A(x)=\frac{N-1}{|x|^2}+\frac{\lambda}{2}$.
Now we claim that $G(r)<0$ for $r>0$. For the sake of contradiction, let us suppose that there exist $0<\bar r$ such that
\begin{equation*}
 G(0)=G(\bar r)=0.
\end{equation*}
By the strong maximum principle, either $G=0$ on $B(0,\bar r)$ or $G<0$ on $B(0,\bar r)$.
In the first case, we infer that $U^*=0$ or $U^*=\alpha^{1/(2\sigma)}$ which is impossible. In the second case, the Hopf Lemma implies 
$G'(\bar r)>0$ which contradicts the fact that $G(r)\le 0$ for every $r>0$. 
Therefore, we have $\partial_r U^*<0$ in $\R^N\setminus\{0\}$ and $\nabla U^*(x) = 0$ if and only if $x=0$.
This means we are in the conditions to apply \cite[Theorem 1.1]{BrZi}. Indeed, by \eqref{grad-eq}, we are in the equality case in the P\'olya-Szeg\"o inequality and $U^*$ is decreasing. Therefore, we conclude that $U^*=U$ a.e. in $\R^N$, i.e. $U$ is radial and radially decreasing.
\end{proof}

\subsection{Exponential decay}\label{expdecay}
In this subsection we investigate the decay of solutions to \eqref{4nls}. Our approach is inspired by \cite{MR2361740} and applies to any solution that goes to $0$ as $|x|\rightarrow \infty$ without requiring any information about the sign. Therefore, it holds for solutions to \eqref{MinL2fixed} and \eqref{MinLpintro}.

\medbreak

\begin{thm}
\label{expdecay}
Let $u$ be a classical solution to
$$
\left\{\begin{array}{ll}
\Delta^2 u -\beta \Delta u +\alpha u= g(x,u)\ in\ \R^N,\\
 \lim_{|x|\rightarrow \infty} u(x)=0,
\end{array}\right.
$$
and assume that 
\begin{enumerate}[(i)]
\item $g(x,u):\R^N \times \R \rightarrow \R$ is measurable in $x$ and continuous in $u$ and 
\begin{equation*}
 \sup_{\substack{0\leq u\leq K\\ x\in \R^N}}|g(x,u)|<\infty\ \ \text{for every}\ K>0.
\end{equation*}
\item There exist a constant $\sigma>0$ and a function $b_1\in L^\infty (\R^N)$ such that $|g(x,u)|\leq b_1 |u|^{1+\sigma}$ when $|u|$ is large enough.
\end{enumerate}
Then, the following assertions hold.
\begin{enumerate}[(1)]
\item Assume that $\beta\neq 2\sqrt{\alpha}$. Then there exists a constant $C>0$ such that, for $|x|$ large enough, we have
$$|u(x)|\leq \dfrac{C}{\sqrt{|\beta^2 - 4\alpha|}} \left\{\begin{array}{ll} e^{-(\frac{\sqrt{\beta - \sqrt{\beta^2 - 4\alpha}}}{\sqrt{2}} -\varepsilon) |x|} & if\ \beta -2 \sqrt{\alpha} >0,\\ e^{-(\frac{\sqrt{2\sqrt{\alpha}-\beta}}{2}-\varepsilon) |x|} & if\ |\beta | - 2\sqrt{\alpha} <0\end{array}\right.$$
for any $\varepsilon >0$.
\item Assume that $\beta= 2\sqrt{\alpha}$, $u>0$ and $u(x),\Delta u(x)\rightarrow 0$ as $|x|\rightarrow \infty$. Then there exists a constant $C>0$ such that, for $|x|$ large enough, we have
$$|u(x)|\leq Ce^{-(\sqrt{\beta}-\varepsilon )|x|},$$
for any $\varepsilon>0$.
\end{enumerate}
\end{thm}

\begin{rmq}\label{RemExpoDecay}
\label{decunif}
Suppose $\alpha$ is fixed and $\beta \rightarrow +\infty$. The previous theorem implies that, for $|x|$ large enough and any $\tilde{\varepsilon}=\varepsilon/\sqrt{\beta}>0$,
\begin{equation}
\label{decayfeb1}
|u(x)|\leq C e^{-\frac{1}{\sqrt{\beta}}(\sqrt{\alpha} -\varepsilon)|x|},
\end{equation}
for some constant $C$ independent of $\beta$. Setting $u(x)=v(\gamma^{\frac{1}{4}} x)$ with $\beta=1/\sqrt{\gamma}$ and assuming that $g(x,u)=g(u)$, we see that $v$ solves
$$\gamma \Delta^2 v -\Delta v +\alpha v=g(v). $$
We deduce from \eqref{decayfeb1} that as $\gamma \rightarrow 0^{+}$, we have, for $|x|$ large enough,
$$|v(x)|\leq C e^{-(\sqrt{\alpha}-\varepsilon)|x| }.$$
Let us remark that this decay is almost optimal: it is well-known, in the case $\gamma=0$, that the following estimate is sharp
$$|v(x)|\leq \dfrac{C}{|x|^{\frac{N-1}{2}}}e^{-\sqrt{\alpha} |x|}.$$
\end{rmq}

We begin by recalling some known facts on the fundamental solutions to the Helmholtz equation. They are solutions to 
\begin{equation}
\label{helmholtz}
-(\Delta +\mu)g_{\mu}(\cdot,y)=\delta_y\ \text{where}\ \mu\in \Cc ,\ y\in \R^N,
\end{equation}
and $\delta_y$ stands for the Dirac mass centered at $y$. By abuse of notation, we fix $y\in \R^N$ and denote $g_\mu (\cdot-y)=g_{\mu}(\cdot,y)$. The functions $g_\mu$ are not uniquely determined but in the following we always choose those which satisfy nice integrability condition, namely we require that $g_\mu \in L^1 (\R^N)$. 

For $N=1$ we get
\begin{equation*}
 g_{\mu}(x)=\dfrac{i}{2\sqrt{\mu}}e^{i\sqrt{\mu}|x|}.
\end{equation*}
In general, for $N>1$ we have
\begin{equation*}
g_{\mu}(x)=\dfrac{i c_N^{\mu} }{|x|^\nu}H_{\nu}^{(1)} (\sqrt{\mu}|x|),\hspace{.3cm}c_N^\mu =\frac{\pi \mu^{\nu/2}}{2(2\pi)^{N/2}}
\end{equation*}
where $\nu =(N-2)/2$ and $H_{\nu}^{(1)}$ is the first Hankel function, see \cite[pg. 76]{Leis86}. It is well-known that
\begin{equation*}
H_\nu^{(1)} (r)= \left(\dfrac{2}{\pi r}\right)^{1/2}\text{exp}\left[i\left(r-\frac{\nu \pi}{2}-\frac{\pi}{4}\right)\right] + o\left(\frac{1}{r^{3/2}}\right)\hspace{.3cm}\text{as}\ \ r\to\infty.
\end{equation*}

\noindent We have the following representation formula, see \cite[pg. 78]{Leis86}.
\begin{thm}
\label{thm1}
(i) Let $f\in L^2 (\R^N) \cap L^\infty (\R^N)$ and
$$u(x)=\int_{\R^N} f(y) g_\mu (x-y)\, dy.$$
Then 
$$-(\Delta +\mu)u= f.$$
(ii) Let $\mu \in \Cc \backslash \R^+$, $u\in L^2 (\R^N)$ and
$$-(\Delta + \mu) u=f \in L^2 (\R^N) \cap L^\infty (\R^N).$$
Then
$$u(x)=\int_{\R^N} f(y) g_{\mu}(x-y)\, dy. $$
\end{thm}
Let us notice that using the exponential decay of $g_\mu$, it is possible to ask for less regularity on $f$. Next, we denote by $G$ the fundamental solution to the operator $\Delta^2 -\beta \Delta +\alpha$, i.e.
\begin{equation}
\Delta^2 G-\beta \Delta G+\alpha G=\delta_0 .
\label{defgreen}
\end{equation}
In the next proposition, we express $G$ in terms of Helmhotz's functions which allows us to get information on its asymptotic behavior.

\begin{prop}
\label{prop1}
Assume that $\beta\neq 2\sqrt{\alpha}$. We have
$$G=\dfrac{1}{\sqrt{\beta^2-4\alpha }} (g_{x_2}-g_{x_1} ),$$
where $g_{x_j}$, $j=1,2$ are the fundamental solutions to \eqref{helmholtz} (as defined previously) with 
\begin{equation}\label{x_1-x_2-formulae}
x_1=\dfrac{-\beta + \sqrt{\beta^2 - 4\alpha}}{2}\hspace{.5cm}\text{and}\hspace{.5cm}x_2=\dfrac{-\beta - \sqrt{\beta^2 - 4\alpha}}{2}. 
\end{equation}
Moreover, for $|x|$ large enough, we have
$$|G (x)| <\dfrac{C}{|\sqrt{\beta^2-4\alpha }|} g_{\tau}(x),$$
for some constant $C>0$, where $0<\tau<(\sqrt{\beta -\sqrt{\beta^2 - 4\alpha}})/2$ if $\beta > 2\sqrt{\alpha}$ and $0 <\tau <(\sqrt{ 2\sqrt{\alpha}-\beta})/4$ if $2\sqrt{\alpha}>|\beta|$.
\end{prop}
\begin{proof}
Taking the Fourier transform of \eqref{defgreen}, we get
$$\widehat{G}= \left(\dfrac{1}{2\pi}\right)^{N/2} \dfrac{1}{(|\xi|^2 -x_1)(|\xi|^2 - x_2)},$$
where $x_1$ and $x_2$ are defined in \eqref{x_1-x_2-formulae}. We have
$$\widehat{G}= \left(\dfrac{1}{2\pi}\right)^{N/2}\dfrac{1}{\sqrt{\beta^2-4\alpha}}\left( \dfrac{1}{|\xi|^2 -x_1}-\dfrac{1}{|\xi|^2 - x_2}\right).$$
Next, taking the Fourier transform of \eqref{helmholtz}, we see that
$$\widehat{g}_{\mu} = -\left(\dfrac{1}{2\pi}\right)^{N/2} \dfrac{1}{|\xi|^2 - \mu}.$$
We deduce from the two previous equalities that
$$\widehat{G}= \dfrac{1}{\sqrt{\beta^2 - 4\alpha }}(\widehat{g}_{x_2}-\widehat{g}_{x_1} ).$$
The first part of the proof follows. We recall that
$$g_{\mu}(x)=\dfrac{i c_N^{\mu} }{|x|^\nu}H_{\nu}^{(1)} (\sqrt{\mu}|x|)$$
and, for $|x|=r$ large, 
\begin{equation*}
H_\nu^{(1)} (r)= \left(\dfrac{2}{\pi r}\right)^{1/2}\text{exp}\left[i\left(r-\frac{\nu \pi}{2}-\frac{\pi}{4}\right)\right] + o\left(\frac{1}{r^{3/2}}\right)\hspace{.3cm}\text{as}\ \ r\to\infty.
\end{equation*}

We have to consider two cases depending on the values of $\alpha$ and $\beta$. First, assume that $\beta - 2 \sqrt{\alpha} > 0$. In this case, we see that $x_1,x_2 \in \Cc \backslash \R^+$ which allows us to 
use the representation formula (see Theorem \ref{thm1}). 
 Then, we have, for $|x|=r$ large enough, for $j=1,2$ and for some constant $C>0$,
$$|g_{ x_j}(x)|\leq \dfrac{C}{r^{\frac{N-1}{2}}} e^{i \sqrt{x_j} r}\leq \dfrac{C}{r^{\frac{N-1}{2}}} e^{-\sqrt{x_j} r}.$$
Thus, we deduce that
$$|G (x)|\leq \dfrac{C}{\sqrt{\beta^2-4\alpha} r^{\frac{N-1}{2}}} e^{-\frac{\sqrt{ \beta -\sqrt{\beta^2 - 4\alpha} } }{\sqrt{2}} r}.$$
Finally, we consider the case $|\beta | - 2\sqrt{\alpha} < 0$. Using the well-known formula
$$\sqrt{a+ib}=A+iB,\ \text{for}\ a,b\in \R,$$
where $A=\pm \dfrac{1}{\sqrt{2}}\sqrt{\sqrt{a^2+b^2}+a}$ and $B=\pm \dfrac{1}{\sqrt{2}}\sqrt{\sqrt{a^2+b^2}-a}$, we obtain, for some constant $A_1\in \R$,
$$i\sqrt{x_j}= i\sqrt{\dfrac{\beta \pm i \sqrt{4\alpha -\beta^2} }{2}}=i A_1 - \dfrac{1}{2}\sqrt{2\sqrt{\alpha} -\beta}. $$
Therefore, proceeding as previously, we deduce that, for $r$ large enough,
$$|G(x)|\leq \dfrac{C}{|\sqrt{\beta^2-4\alpha}| r^{\frac{N-1}{2}}} e^{- \dfrac{1}{2}\sqrt{2\sqrt{\alpha} -\beta}r}.$$
This concludes the proof.
\end{proof}

\noindent We can now state the following representation formula.
\begin{prop}
\label{proprep}
Let $x_1$ and $x_2$ be defined as in \eqref{x_1-x_2-formulae}. Assume that $ x_1,x_2\in \Cc \backslash \R^+$ and let $u$ be a distribution such that
\begin{equation}
\label{eqproprep}
\Delta^2 u-\beta \Delta u+ \alpha u=f\in L^2(\R^N)\cap L^\infty (\R^N).
\end{equation}
Then, we have
$$u(x)=\int_{\R^N} f(y) G(x-y)\, dy.$$
\end{prop}
\begin{proof}
First, we observe that one can rewrite \eqref{eqproprep} as
$$(\Delta +x_1 ) (\Delta +x_2 ) u=(\Delta +x_2)(\Delta +x_1) u =f.$$
We set $u_1=-(\Delta +x_1) u $ and $u_2=-(\Delta +x_2) u $. On the one hand, it is easy to see that 
$$u=\dfrac{1}{\sqrt{\beta^2-4\alpha}} (u_2 -u_1).$$
On the other hand, using Theorem \ref{thm1}, we have
$$ u_1(x)= \int_{\R^N} f(y) g_{x_2}(x-y)\, dy ,$$
and
$$ u_2(x)= \int_{\R^N} f(y) g_{x_1}(x-y)\, dy .$$
Thus, combining the three previous lines, we deduce that
\begin{align*}
u(x)&= \dfrac{1}{\sqrt{\beta^2-4\alpha}} \int_{\R^N} ( g_{x_2}(x-y)- g_{x_1}(x-y)) f(y)\, dy\\
&=\int_{\R^N} G(x-y) f(y)\, dy.
\end{align*}
This concludes the proof.
\end{proof}
We are now in position to prove Theorem \ref{expdecay}.
\begin{proof}[Proof of Theorem \ref{expdecay}.]
We start with the case $\beta\neq 2 \sqrt{\alpha}$. Using Propositions \ref{prop1} and \ref{proprep}, we obtain that
$|u(x)|\leq \dfrac{C}{|\sqrt{\beta^2 - 4\alpha}|} v(x)$ for all $x\in \R^N$ where $$v(x)=\int_{\R^N} g_{\tau}(x-y) |g(y,u(y))|\, dy.$$
Using our assumptions on $g$ and since $\lim_{|x|\to\infty}u(x)=0$, we have that for any $\varepsilon>0$, there exists $R>0$ such that if $|x|\geq R$, then  
\begin{equation*}
\dfrac{|g(x,u(x))|}{v(x)}\leq \dfrac{|u(x)|^{1+\delta}}{v(x)}\leq \varepsilon. 
\end{equation*}

\noindent Next we define $\psi(x)=C_1 e^{-a (|x|-R)}$, where $a=\sqrt{\tau} -\varepsilon$ and $C_1$ is a large positive constant such that $v(x)<C_1$ if $|x|\leq R$. It is easy to see that
$$-\Delta  \psi +a^2\psi\geq 0\hspace{.3cm} \text{in}\ \R^N\backslash \{0\},$$
$$v(x)-\psi(x)< 0\hspace{.3cm} \text{for}\ |x|\leq R,$$
and
\begin{equation}\label{limsup-diff}
 \lim_{|x|\rightarrow \infty}  (v(x) -\psi (x))=0.
\end{equation}
We now show that
\begin{equation*}
 v(x)\leq\psi(x)\hspace{.3cm}\text{for}\ |x|\geq R.
\end{equation*}
If there exists $x_0\in \R^N$ such that $|x_0|\geq R$ and $v(x_0)-\psi (x_0) >0$, then the set
$$\Omega = \{x\in \R^N: |x|>R\ \text{and}\ v(x)>\psi (x)\}$$
is nonempty, $\partial \Omega  \neq \varnothing$ and $\overline{\Omega}\subset \R^N \backslash \overline{B(0,R)}$. For $x\in\Omega$, we have
$$\Delta (v(x) -\psi (x)) = a^2 (v(x)-\psi(x))>0.$$
Applying \eqref{limsup-diff} and the maximum principle, we get that $v(x)-\psi(x)\leq 0$ for $x\in\Omega$, which contradicts the definition of $\Omega$. Therefore, we obtain
$$|u(x)|\leq \dfrac{C}{|\sqrt{\beta^2-4\alpha}|} v(x)\leq C e^{-a |x|},$$
where $a=\sqrt{\tau} -\varepsilon$ ($\tau$ defined as in Proposition \ref{prop1}).

\medbreak

Assume now $\beta=2\sqrt{\alpha}$. In this case, let us recall that we assume that $u>0$ and $u(x),\Delta u(x)\rightarrow 0$ as $|x|\rightarrow \infty$. Since $u(x)\rightarrow 0$ as $|x|\rightarrow \infty$, we can choose $R>0$ such that $|u(x)|<\tilde{\delta}$ for some constant $\tilde{\delta}>0$ and $|x|>R$. Thus, we have
\begin{equation}
\label{eqcrit}
\Delta^2 u - \beta \Delta u + (\alpha -\tilde{\delta}) u \geq 0.
\end{equation}
Since $\tilde{\delta}>0$, we have $\beta^2\geq 4(\alpha -\delta)=4\tilde{\alpha}$. Therefore, we can rewrite \eqref{eqcrit} as
$$\left\{\begin{array}{ll}
-\Delta u + \lambda_1 u &=v\\
-\Delta v +\lambda_2 v & \geq 0,
\end{array}\right.$$
where $\lambda_1,\lambda_2>0$ are such that $\lambda_1 +\lambda_2=\beta$ and $\lambda_1 \lambda_2 =\tilde{\alpha}$. Let us notice that the maximum principle implies that $v\geq 0$.
Next, we recall that the solution to 
$-\Delta u+\gamma_1 u=\delta_0$, $\gamma_1 >0$ is given by
$$V_{\gamma_1}=x^{-N/2} [C_1 J_1 (i\sqrt{\gamma_1} x) + C_2 Y_1 (i\sqrt{\gamma_1} x) ],$$
where $J_1$ and $Y_1$ are the Bessel functions of first and second kind, respectively. Using the asymptotic behaviour of these functions, we deduce that, for $|x|$ large enough,
$$|V_{\gamma_1} (x)|\leq \dfrac{C}{|x|^{\frac{N-1}{2}}}e^{-\sqrt{\gamma_1} |x|}.$$
Next, we choose $C_1,C_2$ such that $V_{\mu} (x)>u(x)$ for $|x|=R$. Since $u(x)\rightarrow 0$ as $|x|\to\infty$, and using the maximum principle, we get
$$|u(x)|\leq \dfrac{C}{|x|^{\frac{N-1}{2}}}e^{-\sqrt{\mu} |x|}.$$
We notice that we can choose $\mu=\beta-\varepsilon$, for some $\varepsilon$ as small as we want.
\end{proof}
\begin{rmq}
Assume that $\beta\geq 2 \sqrt{\alpha}$ and let $u$ be the ground state solution to $\eqref{eq4beta}$. We know that $u$ is radial, positive and that $u(x),\Delta u(x)\rightarrow 0$ as $|x|\rightarrow \infty$. Therefore, in view of the previous results, there exist $a,C>0$ such that, for $|x|\geq R$,
$$|u(x)|\leq C e^{-a|x|}.$$
Since $\beta\geq 2 \sqrt{\alpha}$, we can write
$$
\left.
\begin{cases}
-\Delta u + \dfrac{\beta}{2} u &=v\\
-\Delta v +\dfrac{\beta}{2} v & = |u|^{2\sigma}u +\left(\dfrac{\beta^2}{4}-\alpha\right) u.
\end{cases}
\right.$$
As previously, since $ |u|^{2\sigma}u +(\beta^2/4-\alpha) u>0$, one can show that there exist $a_1,C>0$ such that, for $|x|\geq R$,
$$|v(x)|\leq C e^{-a_1|x|}.$$
So, we see that $v^\prime $ satisfies $(r^{n-1}v^\prime)^\prime \leq r^{n-1} e^{-\min \{a,a_1\} r} $. Integrating, we deduce that, for $|x|\geq R$,
$$|v^\prime (x)|\leq C e^{-b |x|},$$
for some constant $b>0$. The same kind of argument also implies that $u^\prime$ has an exponential decay. 
\end{rmq}

\subsection{Nondegeneracy in dimension one} 
In this section we investigate the nondegeneracy of the solutions to 
\begin{equation}
\label{4NLS}
 u'''' -\beta  u^{\prime \prime}+\alpha u= |u|^{2\sigma}u,\ \lim_{|t|\rightarrow \infty}(u,u',u'',u''')(t)=0.
\end{equation}
Uniqueness (up to translations) can be deduced by adapting the arguments of \cite[Section 2]{buffonitoland} and \cite{amick}. Our aim in this subsection is to prove the following result. 
\begin{thm}
\label{thmnondegen}
Assume that $\beta\geq 2\sqrt{\alpha}$. Let $U$ be the unique solution (up to translations) to \eqref{4NLS}. If $v$ solves the linearized equation around $U$
\begin{equation}\label{eq:1dlin}
v'''' -\beta  v^{\prime \prime}+\alpha v=(2\sigma+1) |U|^{2\sigma}v,
\end{equation}
and $\displaystyle\lim_{|t|\rightarrow \infty}(v,v',v'',v''')(t)=(0,0,0,0)$, then $v\in \left<U^\prime\right>$.
\end{thm}
We make use of the arguments developped in \cite[Section 2]{buffonitoland}. For $\beta\geq 2\sqrt{\alpha}$ we define $\lambda_1$ and $\lambda_2$ as
\begin{equation}\label{LambdaDecomp}
 \lambda_1\lambda_2=\alpha\hspace{.3cm}\text{and}\hspace{.3cm}\lambda_1+\lambda_2=\beta
\end{equation}
and then $\lambda_1$ and $\lambda_2$ are positive. Thus \eqref{4NLS} can be factorised as
$$\left.
\begin{cases}
u^{\prime \prime}-\lambda_1 u&=w, \\
w^{\prime \prime}-\lambda_2 w&=|u|^{2\sigma}u.
\end{cases}\right.$$
\begin{prop}
\label{compa}
Let $u$ and $v$ be solutions to \eqref{4NLS} on $[0,\infty)$
with
$$u(0)\geq v(0)>0,\ u^\prime (0)\geq v^\prime (0),$$
$$( u^{\prime \prime}-\lambda_1 u)(0)\geq ( v^{\prime \prime}-\lambda_1 v)(0),\ ( u^{\prime \prime}-\lambda_1 u)^\prime (0)\geq ( v^{\prime \prime}-\lambda_1 v)^\prime (0), $$
where $\lambda_1$ is defined in \eqref{LambdaDecomp}. If $v(t)\geq 0$, $t\in [0,\infty)$, then $u\equiv v$ or $u(t)\rightarrow \infty$ as $t\rightarrow \infty$.
\end{prop}

\begin{proof}
Let $u$ and $v$ satisfy the assumptions in the proposition and suppose that $u\neq v$. Let $k$ be the smallest integer such that $u^{(k)}(0)\neq v^{(k)}(0)$. By assumptions we have $u^{(k)}(0)>v^{(k)}(0)$. Thus, there exists $T>0$ such that
$$
u(t)>v(t)\geq 0\hspace{.3cm}\text{for all}\  t\in (0,T).
$$
Now we set
$$\varphi(t)= u^{\prime \prime}-\lambda_1 u\hspace{.3cm}\text{and}\hspace{.3cm}\psi(t)= v^{\prime \prime}-\lambda_1 v.$$
By assumptions, we have
$$\varphi(0)- \psi(0)\geq 0,\ \varphi^\prime(0)-\psi^\prime(0)\geq 0, $$
and
$$ (\varphi - \psi)^{\prime \prime} (t)-\lambda_2 (\varphi-\psi )(t)= |u|^{2\sigma}u(t)-|v|^{2\sigma}v(t) >0\hspace{.3cm}\text{on}\hspace{.3cm}(0,T).$$
Thus
$$\varphi(t)-\psi(t)> 0\hspace{.3cm}\text{on}\hspace{.3cm}(0,T),$$
that is
\begin{equation*}
 (u-v)^{\prime \prime}(t)- \lambda_1 (u-v)(t)> 0\hspace{.3cm}\text{on}\hspace{.3cm}(0,T).
\end{equation*}
Since
\begin{equation*}
 (u-v)(0)\geq 0\hspace{.3cm}\text{and}\hspace{.3cm}(u-v)'(0)\geq 0
\end{equation*}
we obtain
$$u(t)>v(t)\geq 0\hspace{.3cm}\text{on}\hspace{.3cm}(0,T],$$
and by using the continuity we have $u(t)>v(t)$ on $(0,T+\varepsilon)$ for $\varepsilon>0$ small enough. Thus
\begin{equation*}
 \sup\{T:u(t)>v(t)\ \text{for all}\ t\in(0,T)\}=\infty
\end{equation*}
and 
\begin{align*}
 (\varphi-\psi)^{\prime \prime}(t)-\lambda_2 (\varphi-\psi)(t)&\geq 0\hspace{.3cm}\text{on}\hspace{.3cm}(0,\infty),\\
 (u-v)^{\prime \prime}(t)-\lambda_1 (u-v)(t)&\geq 0\hspace{.3cm}\text{on}\hspace{.3cm}(0,\infty).
\end{align*}
Since $u\not\equiv v$, it follows that
$$u(t)-v(t)\rightarrow \infty\hspace{.3cm}\text{as}\hspace{.3cm}t\rightarrow \infty.$$
Finally, since $v(t)\geq 0$ we conclude that $u\rightarrow \infty$ as $t\to\infty$.
\end{proof}

\begin{prop}
\label{compa1}
Suppose $v$ is a solution to \eqref{4NLS} with $v(0)>0> v^{\prime \prime}(0) $ and $v^\prime (0)=v^{(3)}(0)=0$, such that $v(t)\rightarrow 0$ as $t\rightarrow \infty$. Suppose $u$ is another solution satisfying
$$k\varepsilon \geq u^{\prime \prime}(0) - v^{\prime \prime}(0)\geq (1+a)\lambda_1 (u(0)-v(0))\geq (1+a)\lambda_1 \varepsilon, $$
$$|u^\prime (0)|+|u^{(3)}(0)|\leq b \varepsilon^2,$$
where $k,a,b$ are positive constants and $\lambda_1$ is defined in \eqref{LambdaDecomp}. Then $u(t)$ does not go to $0$ as $t\rightarrow \infty$.
\end{prop}

\begin{proof}
First let us recall that by \cite{amick}, there exists a unique even solution to \eqref{4NLS} which can be obtained by considering the minimization problem \eqref{MinLpintro}. Thus, we deduce from Theorem \ref{thm1.1bonnas} that $v>0$. Using the Taylor expansion, we have
\begin{align*}
u(t)-v(t)&\geq \varepsilon - b\varepsilon^2 t -Ct^2\\ 
u^\prime (t) -v^\prime (t)&\geq -b \varepsilon^2 + (1+a)\lambda_1\varepsilon t-\dfrac{1}{2}b\varepsilon^2 t^2 -C t^3\\
(u^{\prime \prime}- \lambda_1 u)(t)-(v^{\prime \prime}- \lambda_1 v)(t)&\geq a \lambda_1 \varepsilon - b\varepsilon^2 t-Ct^2\\
(u^{\prime \prime \prime}- \lambda_1 u^{\prime})(t)-(v^{\prime \prime \prime}- \lambda_1 v^{\prime})(t)&\geq -b \varepsilon^2+a\alpha\varepsilon t- \gamma_1 \varepsilon^2 t^2-Ct^3,
\end{align*}
for some $\gamma_1,C>0$. For the last inequality we have used that $u(0)>v(0)>0$ and that for $t=0$ we have
\begin{align*}
(u'''' - \lambda_1 u^{\prime \prime})&= \lambda_2 (u^{\prime \prime}-\lambda_1 u)+u^{2\sigma+1}\\
&\geq\lambda_2 (v^{\prime \prime}-\lambda_1 v) +a\alpha (u-v)+u^{2\sigma+1}\\
&\geq \lambda_2 v^{\prime \prime}-\lambda_2\lambda_1 v + a\alpha \varepsilon +u^{2\sigma+1}\\
&=(v''''-\lambda_1 v^{\prime \prime})+a\alpha\varepsilon+ (u^{2\sigma+1}-v^{2\sigma+1})\\
&>(v''''-\lambda_1 v^{\prime \prime})+a\alpha \varepsilon.
\end{align*}
We have
$$u(t)-v(t)\geq 0\hspace{.3cm}\text{and}\hspace{.3cm}(u^{\prime \prime}-\lambda_1 u)(t)-(v^{\prime \prime}-\lambda_1 v)(t)\geq 0,$$
for all $t\in [0,\Gamma (\varepsilon)]$ where 
\begin{equation*}
\Gamma (\varepsilon)=\sqrt{\varepsilon}\left(\sqrt{\dfrac{\min \left\{1,a \lambda_1 \right\}}{C}}-\dfrac{b \varepsilon^{3/2}}{2C}\right). 
\end{equation*}
Let $\tau(\varepsilon)=\varepsilon^{2/3}\in [0,\Gamma (\varepsilon)]$. One can show that
$$(u^{\prime \prime \prime}- \lambda_1 u^{\prime})(\tau(\varepsilon))-(v^{\prime \prime \prime}- \lambda_1 v^{\prime})(\tau (\varepsilon))>0,$$
and
$$(u^\prime -v^\prime)(\tau(\varepsilon))>0$$
for $\varepsilon>0$ sufficiently small. Hence, 
by Proposition \ref{compa} and since $v>0$, we conclude that $u\not\to 0$ as $t\to\infty$.
\end{proof}

\begin{prop}
\label{compa2}
Let $v$ be the solution defined in the previous proposition. Suppose that $u$ is a solution to \eqref{4NLS} with
$$4a\varepsilon\geq u^{\prime \prime \prime}(0)-\lambda_1 u^\prime (0)\geq u^\prime (0)\geq a \varepsilon,$$
$$|u(0)-v(0)|+|u^{\prime \prime}(0)-v^{\prime \prime}(0)|\leq b \varepsilon^2,$$
where $a,b$ are positive constants.  
Then $u(t)$ does not go to $0$ as $t\rightarrow \infty$.
\end{prop}
\begin{proof}
By using the Taylor expansion, we get
\begin{align*}
 u(t)-v(t)&\geq - b\varepsilon^2 +a\varepsilon t-\dfrac{1}{2}b\varepsilon^2t^2 -Ct^3\\
u^\prime (t) -v^\prime (t)&\geq a\varepsilon-b\varepsilon^2 t -C t^2\\
(u^{\prime \prime}- \lambda_1 u)(t)-(v^{\prime \prime}- \lambda_1 v)(t)&\geq -2b \varepsilon^2+a \varepsilon t- K\varepsilon^2 t^2-Ct^3\\
(u^{\prime \prime \prime}- \lambda_1 u^{\prime})(t)-(v^{\prime \prime \prime}- \lambda_1 v)(t)&\geq a\varepsilon - K \varepsilon^2 t-Ct^2,
\end{align*}
for some $K,C>0$. We have
$$(u^{\prime \prime \prime}- \lambda_1 u^{\prime})(t)-(v^{\prime \prime \prime}- \lambda_1 v^{\prime})(t)\geq 0,$$
and
$$(u^\prime -v^\prime)(t)\geq 0,$$
for all $t\in [0,\Gamma (\varepsilon)]$ where 
\begin{equation*}
\Gamma (\varepsilon)=\sqrt{\varepsilon} \left(\sqrt{\dfrac{\min \{1,a \lambda_1\} }{C}}-\dfrac{\min \{b,K\}}{2C} \varepsilon^{3/2}\right). 
\end{equation*}
Let $\tau(\varepsilon)=\varepsilon^{2/3}\in [0,\Gamma (\varepsilon)]$. One can show that
$$u(\tau(\varepsilon))-v(\tau(\varepsilon))> 0\hspace{.3cm}\text{and}\hspace{.3cm}(u^{\prime \prime}-\lambda_1 u)(\tau(\varepsilon))-(v^{\prime \prime}-\lambda_1 v)(\tau(\varepsilon))>0,$$
for $\varepsilon>0$ sufficiently small. Therefore, by Proposition \ref{compa} and since $v>0$, we conclude that $u$ does not go to $0$ as $t\to\infty$.
\end{proof}

\begin{proof}[Proof of Theorem \ref{thmnondegen}] 
Assume $U$ achieves its maximum at $0$. Then, by uniqueness, $U$ is even so that $U'(0)=0$ and $U'''(0)=0$. Observe that since  $\beta\geq 2\sqrt{\alpha}$, $0$ is a hyperbolic equilibrium. Let 
$$W^s = \{z\in \R^4 \mid \lim_{t\to\infty}U(t;z)=0\}
\ \text{ and }\ W^u=\{z\in\R^4\mid \lim_{t\to-\infty}U(t;z)=0\},$$
where $U(t;z)$ is the unique solution of the Cauchy problem
$$u'''' -\beta  u^{\prime \prime}+\alpha u= |u|^{2\sigma}u,\quad (u,u',u'',u''')(0)=z.$$
As $U$ is a homoclinic solution and the equation is autonomous, $(U(t),U'(t),U''(t),U'''(t))\in W^s \cap W^u$ for every $t\in \R$. For every $\bar t\in \R$, the tangent spaces of the stable and unstable manifolds evaluated at $U(\bar t)$ are given by 
$$T_{U(\bar t)}W^s = \{z\in \R^4 \mid \lim_{t\to\infty}V(t;\bar t, z)=0\}
\ \text{ and }\ T_{U(\bar t)}W^u=\{z\in\R^4\mid \lim_{t\to-\infty}V(t;\bar t, z)=0\},$$
where 
$V(t;\bar t,z)$ is the unique solution of the linearized equation \eqref{eq:1dlin} such that 
$(v,v',v'',v''')(\bar t)=z$. See for instance \cite{HOMBURG2010379}.

We claim that $W^s$ and $W^u$ intersects transversally at $\bold{U}(0)=(U(0),0,U^{\prime \prime}(0),0)$. Observe first that 
since the Hamiltonian energy is conserved and $U$ is asymptotic to the equilibrium, $U$ lies in the level $0$ of energy, i.e.  
$$H(U(t),U'(t),U''(t),U'''(t)) = \left(U^\prime\left(U^{(3)}+\dfrac{\beta}{2}U^\prime\right)+\dfrac{\alpha}{2}U^2-\dfrac{1}{2\sigma+2}U^{2\sigma+2}-\dfrac{1}{2}U^{\prime \prime 2}\right)(t)=0,\hspace{.2cm}\forall\, t\in\R.$$ As a consequence, if $z\in T_{\bold{U}(0)}W^s\cap T_{\bold{U}(0)} W^u$, then 
$$dH(U(0),0,U''(0),0)z=(\alpha-U(0)^{2\sigma})U(0)z_1 -U''(0) z_3  =0.$$ Since $U$ is even, the equation \eqref{eq:1dlin} is reversible and therefore 
\begin{equation*}
(a,b,c,d)\in T_{\bold{U}(0)}W^s\hspace{.2cm}\text{if and only if}\hspace{.2cm}(a,-b,c,-d)\in T_{\bold{U}(0)} W^u. 
\end{equation*}
Observe also that since $U'$ is a solution of \eqref{eq:1dlin}, we have that $T_{\bold{U}(0)}W^s\cap T_{\bold{U}(0)} W^u$ contains at least $(0,U''(0),0,\beta  U''(0)-\alpha U(0) + U(0)^{2\sigma})$.

Assume now by contradiction that the tangent spaces of the stable and unstable manifolds coincide in a two dimensional space. There are then two possibilities. Either $T_{\bold{U}(0)}W^s\cap T_{\bold{U}(0)} W^u$ is given by 
$$\big\{a(0,1,0,0)+b(0,0,0,1) \mid a,b\in \R \big\}$$
or
$$
\Big\{a\left(-U^{\prime \prime}(0),0,U^{2\sigma +1}(0)-\alpha U(0),0\right) +b\left(0,U''(0),0,\beta U''(0)-\alpha U(0)+U^{2\sigma+1}(0) \right) \mid a,b\in \R \Big\}. 
$$
In the first case, taking $a=1,b=3$ and considering the point of the stable manifold
$$(u,u^\prime,u^{\prime \prime},u^{(3)})(0)=\Big(U(0)+o(\varepsilon^2),\varepsilon+o(\varepsilon^2), U^{\prime \prime}(0)+o(\varepsilon^2),3\varepsilon+o(\varepsilon^2) \Big),$$
we are in the settings of Proposition \ref{compa1} and therefore $u(t)$ does not go to zero as $t\to\infty$. This is a contradiction with the fact that $(u,u^\prime,u^{\prime \prime},u^{(3)})(0)\in W^s$. In the second case, let $a=-1/U^{\prime \prime}(0)$, $b=0$, and consider the point
\begin{equation*}
(u,u^\prime,u^{\prime \prime},u^{(3)})(0) =\Big(U(0)+\varepsilon+o(\varepsilon^2),o(\varepsilon^2), U^{\prime \prime}(0)-\varepsilon \dfrac{(U^{2\sigma +1}(0)-\alpha U(0))}{U^{\prime \prime}(0)},o(\varepsilon^2)\Big). 
\end{equation*}
Observe that, for some $\eta >0$,
$$(u^{\prime \prime}-U^{\prime \prime})(0)\geq \varepsilon (1+\eta)+o(\varepsilon^2) \geq (1+\eta)(u-U)(0). $$
Thus, Proposition \ref{compa2} applies and yields yet a contradiction. Hence the claim holds and 
$$T_{\bold{U}(0)}W^s\cap T_{\bold{U}(0)} W^u = \Big<\left(0,U''(0),0,\beta U''(0)-\alpha U(0)+U^{2\sigma+1}(0) \right) \Big>.$$
The nondegeneracy now immediately follows. Indeed, assume that the linearized equation \eqref{eq:1dlin} admits two linearly independent solutions $v_1$ and $v_2$. Then $v_1(0)$ and $v_2(0)$ are linearly independent vectors of $T_{\bold{U}(0)}W^s\cap T_{\bold{U}(0)} W^u$ which is a contradiction.

\end{proof}

\section{The effect of a small fourth order dispersion}\label{sec:smalldis} 
In this section we are interested in the asymptotic behaviour of solutions $u_{\gamma}$ to \eqref{MinL2fixed} obtained in Section $2$ and the ground state solutions as $\gamma$ goes to zero. We aim at proving that, in this case, the solutions $u_{\gamma}$ converge to a ground state solution to \eqref{nlsalpha}. 
We assume without loss of generality that $\beta=1$.

When $0<\sigma<2/(N-2)$ if $N\geq 3$ or $\sigma>0$ otherwise, it is well known, see for instance \cite{Kwong}, that if $u\in H^1(\R^N)$ is positive and solves
\begin{equation}\label{eqforuni}
 -\Delta u+ u=|u|^{2\sigma}u\hspace{.2cm}\text{in}\ \R^N,\ 
\end{equation}
then $u(x)=u_0(x-x_0)$ for some $x_0\in\R^N$ where $u_0$ is the unique positive radial solution to \eqref{eqforuni}. Namely, uniqueness up to translations holds for positive solutions to \eqref{eqforuni}. 
Moreover $u$ is nondegenerate which means that if $v\in H^1(\R^N)$ solves 
the linearized equation
\begin{equation}\label{nls-linearized}
- \Delta v + v=(2\sigma +1) |u|^{2\sigma}v,
\end{equation} 
then there exists $\xi\in \R^N$ such that $v(x) = \xi\cdot\nabla u(x)$. In particular, there is no nontrivial solution in $H^1_{rad}(\R^N)$ of the linearized equation 
\begin{equation}\label{nls-linearized0}
- \Delta v + v=(2\sigma +1) |u_0|^{2\sigma}v.
\end{equation} 
We refer to Kabeya and Tanaka \cite{KaTa}, see also \cite{Kwong}.
Consequently, all positive $H^1$ solutions to \eqref{nlsalpha} are given by $u_\alpha(x;x_0)=\displaystyle\alpha^{\frac{1}{2\sigma}}u_0(\sqrt{\alpha}(x-x_0))$, where $x_0\in\R^N$, and these solutions are nondegenerate. In what follows we denote by $w_\mu$ the unique (up to sign) radial minimizer (i.e. its maximum is set at the origin) of $E_0$ under the constraint $\|w_\mu\|_{L^2}^2=\mu$ that is
\begin{equation*}
 E_0(w_{\mu})=\inf_{u\in M_\mu}E_0(u)=I_{0,1}(\mu).
\end{equation*}
The nondegeneracy of $w_\mu$ suggests that uniqueness holds also for \eqref{MinL2fixed} when $\gamma$ is small. In this section we indeed use the implicit function theorem to prove uniqueness and nondegeneracy for \eqref{MinL2fixed} as $\gamma \rightarrow 0$. 
To this aim, we first prove that if $\gamma:=\gamma_k\to 0$, then, for any sequence of minimizers $(u_{k,\gamma_k})_k:=(u_k)_k$ of $E_{\gamma_k,1}$, there exists a sequence $(y_k)_k\subset \R^N$ such that $u_k (\cdot-y_k)$ converges strongly in $H^1$ to $w_\mu$. Then, by using elliptic regularity, 
we show that the convergence is actually strong in $H^2$. 
Using this $H^2$-convergence, we finally prove that the solutions $u_k$ are nondegenerate for large $k$ which in turn implies uniqueness up to translations.
With the same arguments, we can also deal with the minimizers of $J_{\gamma_k,1,\alpha}$.

\subsection{Standing waves with a prescribed mass}
We begin with solutions to \eqref{MinL2fixed}. First, we deduce the $H^1$-convergence from the minimality of the sequence. 
\begin{prop}[$H^1$-convergence]\label{MiniConvergenceH1}
Assume $0<\sigma<2/N$. If $\gamma_{k}\to 0$ and $(u_{k})_{k}$ is a sequence such that $E_{\gamma_{k},1}(u_{k})=I_{\gamma_{k},1}(\mu)$, 
 then, up to a subsequence and a sequence $(y_k)_k\subset\R^N$, 
$u_{k}(\cdot-y_k)\to w_{\mu}$ strongly in $H^1$. Moreover, we have 
$$\gamma_k\int_{\R^N}|\Delta u_{k}|^2\, dx\to 0\hspace{.3cm}\text{as}\hspace{.2cm}k\to\infty.$$
\end{prop}
\begin{proof}
First, we show that there exists $C>0$ such that for every $\gamma>0$, we have 
\begin{equation*}
I_{0,1}(\mu)\leq I_{\gamma,1}({\mu})\leq I_{0,1}(\mu)+ C\gamma. 
\end{equation*}
Indeed, let $w_\mu$ be a minimizer of $E_0$ under the constraint $\|w_\mu\|_{L^2}^2=\mu$. Note that by elliptic regularity we know that $w_\mu\in H^2(\R^N)$.  Thus
\begin{equation*}
 I_{\gamma,1}({\mu})\leq E_{\gamma,1}(w_\mu)=\frac{\gamma}{2}\int_{\R^N}|\Delta w_\mu|^2\, dx+ E_{0}(w_\mu)\leq\gamma C+I_{0,1}(\mu).
\end{equation*}
On the other hand, by taking a minimizer $u_{\gamma}$ for $I_{\gamma,1}({\mu})$, we obtain
\begin{equation*}
I_{\gamma,1}({\mu})= E_{\gamma,1}(u_{\gamma})=\frac{\gamma}{2}\int_{\R^N}|\Delta u_{\gamma}|^2\, dx+ E_{0}(u_{\gamma})\geq I_{0,1}(\mu). 
\end{equation*}

\medbreak

Now, take a sequence $(\gamma_k)_k\subset\R^+_0$ such that $\gamma_k\to 0$ as $k\to\infty$. Let $(u_k)_k$ be a minimizer for $I_{\gamma_k,1}(\mu)$. Then
\begin{equation*}
\frac{1}{2}\int_{\R^N}|\nabla u_k|^2\, dx-\frac{1}{2\sigma+2}\int_{\R^N}|u_k|^{2\sigma+2}\, dx\leq I_{\gamma_k,1}(\mu)\leq C\gamma_{k}+I_{0,1}(\mu)\to I_{0,1}(\mu),
\end{equation*}
which shows that $(u_k)_k$ is a minimizing sequence for $I_{0,1}(\mu)$. Thus, $(u_k)_k$ is a bounded sequence in $H^1(\R^N)$ and extracting a subsequence if necessary, we can assume that $u_k\rightharpoonup u$ weakly converges in $H^1(\R^N)$. 
Since $I_{0,1}(\mu)<0$, one can argue with standard arguments, see for instance the $H^1$-subcritical case of the proof of Theorem \ref{Compact-Min-Sol}, to show that compactness holds up to translations. Namely, for some sequence $v_k(x):=(u_k(x-y_k))_k$, we have $v_k\to u$ in $L^2(\R^N)$ and $v_k\to u$ in $L^{2\sigma+2}(\R^N)$, both convergences being strong. 
Hence, from the weak convergence in $H^{1}(\R^N)$, we infer
\begin{align*}
 I_{0,1}(\mu)&\leq\liminf_{k\to\infty}\left(\frac{1}{2}\int_{\R^N}|\nabla v_k|^2\, dx-\frac{1}{2\sigma+2}\int_{\R^N}|v_k|^{2\sigma+2}\, dx\right)\\
 &\leq\limsup_{k\to\infty}\left(\frac{1}{2}\int_{\R^N}|\nabla v_k|^2\, dx-\frac{1}{2\sigma+2}\int_{\R^N}|v_k|^{2\sigma+2}\, dx\right)=I_{0,1}(\mu), 
\end{align*}
which implies that the convergence is strong in $H^1$ and therefore $u$ is a minimizer. By adjusting the sequence $(y_k)_k$, we can assume $u=w_\mu$. 
\end{proof}

Next, we use the equation to improve the convergence. 

\begin{prop}[$H^2$-convergence]\label{uniflimit}
Assume $0<\sigma<2/N$. If $\gamma_{k}\to 0$ and $(u_{k})_{k}$ is a sequence such that $E_{\gamma_{k},1}(u_{k})=I_{\gamma_{k},1}(\mu)$, 
then, up to a subsequence and a sequence $(y_k)_k\subset\R^N$,  
$u_{k}(\cdot-y_k)\to w_{\mu}$ strongly in $H^2(\R^N)\cap W^{1,p}(\R^N)$ for every $p\ge 2$.
\end{prop}
\begin{proof}
By Proposition \ref{MiniConvergenceH1}, there exist $(y_k)_k\subset\R^N$ and a subsequence such that $v_k = u_{k}(\cdot-y_k)\to w_{\mu}$ strongly in $H^1$.  

\medbreak

\noindent\emph{Step one: Uniform boundedness.} We argue similarly as in the proof of Theorem \ref{thmradialsignchanging}. However, due to the fact that $\gamma_k \rightarrow 0$, we need to be more careful. Consider the Euler-Lagrange equation for $v_k$\begin{equation}\label{Euler-Lag}
\gamma_k\Delta^2 v_k-\Delta v_k+\alpha_k v_k=|v_{k}|^{2\sigma}v_{k}. 
\end{equation}
Let us denote $\zeta_k=-\gamma_k\Delta v_k$. Then $\zeta_k$ solves
\begin{equation}\label{eqvk}
-\Delta \zeta_k+\frac{1}{\gamma_k}\zeta_k=w_k, 
\end{equation}
where $w_k=|v_{k}|^{2\sigma}v_{k}-\alpha_k v_k$. Observe that $w_k\in L^p_{\loc}(\R^N)$ for $1\le p\le \frac{2N}{(N-2)(2\sigma+1)}$ since $v_k\in L^1_{\loc}(\R^N)$ and by Sobolev inequality we know that $v_k$ is bounded in $L^{2N/(N-2)}(\R^N)$. Now, since $\zeta_k\to 0$ in $L^2(\R^N)$, we have $\zeta_k\in L^q_{\loc}(\R^N)$ for $1\le q\le 2$. All these bounds are locally uniform (e.g. on unit cubes). From local elliptic regularity, see e.g. \cite[Chapter $5$, Theorem $5$]{kry}, we now deduce a locally uniform $W_{\loc}^{2,r}$-bound on $\zeta_k$ from \eqref{eqvk}, with $r=\min(\frac{2N}{(N-2)(2\sigma+1)},2)$. Then, going back to \eqref{Euler-Lag}, we write
\begin{equation}\label{Eq-u-and-v}
 -\Delta v_k+\alpha_k v_k=|v_{k}|^{2\sigma}v_{k}+\Delta \zeta_k.
\end{equation}
Since the right-hand side is uniformly bounded in $L^r_{\loc}(\R^N)$, we deduce now a locally uniform $W_{\loc}^{2,r}$-bound on $v_k$ again by local elliptic regularity estimates. Since $N/m<r<N/(m-1)$ for some $m\in \N$, bootstrapping the argument $m$ times eventually increases the regularity to $W_{\loc}^{2,s}$ for some $s>N$. Therefore we deduce a uniform $L^\infty$-bound on $v_k$ and $\nabla v_k$. 

\medbreak

\noindent\emph{Step two: Convergence in $W^{1,p}(\R^N)$.} Since the sequence $(v_k)_k$ strongly converges in $H^1(\R^N)$ and $$\sup_k(\|v_k\|_{L^\infty},\|\nabla v_k\|_{L^\infty})<\infty ,$$ the claim is clear.

\medbreak

\noindent\emph{Step three: Convergence of the Lagrange multiplier.}
Let
$$\alpha_k (\mu) = \frac{1}{\mu}\left(-2 E_{\gamma_k,1}(v_k) + \dfrac{\sigma}{\sigma+1}\int_{\R^N}  |v_k |^{2\sigma +2}\, dx\right). $$
We claim that there exist two strictly positive constants $c_1,c_2$ independent of $k$ such that 
\begin{equation}
\label{smalagraunif}
c_1 \leq \alpha_k (\mu) \leq c_2.
\end{equation}
 Indeed, by Remark \ref{lagrunifest}, we have that $ \alpha_k \leq M$ for some constant $M$ independent of $k$. On the other hand, using the Derrick-Pohozaev identity and the strong convergence of $v_k \to w_\mu$ in $H^1 (\R^N)$, one can show that
$$\liminf_{k\to\infty}\alpha_k (\mu) \geq \dfrac{1}{\mu} \left(1-\dfrac{\sigma N}{2\sigma +2}\right) \int_{\R^N} {w_\mu}^{2\sigma +2}\, dx.$$ 
This establishes the claim. 
Moreover, using once more the Derrick-Pohozaev identity and the strong convergence of $v_k \to w_\mu$ in $H^1 (\R^N)\cap L^{2\sigma +2}(\R^N)$, one can show that
$$\alpha_k (\mu)= \dfrac{1}{\mu} \left(-\dfrac{1}{2}\int_{\R^N}|\nabla v_k|^2 dx + \left(1-\dfrac{\sigma N}{2(2\sigma +2)}\right) \int_{\R^N} v_k^{2\sigma +2}dx\right) \rightarrow \alpha_0.$$

\medbreak 

\noindent\emph{Step four: Strong convergence in $H^{2}$.} Consider again the Euler-Lagrange equation \eqref{Euler-Lag} for $v_k$.
Denote again $\zeta_k=-\gamma_k\Delta v_k$. Then $\zeta_k$ solves \eqref{eqvk}.	
Observe that $w_k\in H^1(\R^N)$ since $v_k$ is bounded, $\nabla v_k\in L^2(\R^N)$ and Proposition \ref{MiniConvergenceH1} implies $\zeta_k\to 0$ in $L^2(\R^N)$. We therefore infer from \eqref{eqvk} that $(\zeta_k)_k$ is bounded in $H^3(\R^N)$ by a constant that does not depend on $k$, see for instance \cite[Chapter 1, Theorems 6.4 \& 6.5]{kry}. Now, going back to \eqref{Eq-u-and-v}, we use the $H^3$-bound on $(\zeta_k)_k$ and the fact that $|v_k|^{2\sigma}v_{k}$ is bounded uniformly in $H^1$ to conclude that $(v_k)_k$ is bounded in $H^3 (\R^N)$. Finally, by interpolation, using the strong convergence in $H^1$ and the uniform bound in $H^3$, we deduce the strong convergence of $v_k \to w_\mu$ in $H^2$.
\end{proof}

\begin{rmq}\label{rem:last}
Arguing as in the previous proof, one can show that 
$u_{k}(\cdot-y_k)\to w_{\mu}$ locally in $C^{1,\alpha}$ for every $\alpha\in [0,1)$.
Indeed, since $v_k$ and $\nabla v_k$ are uniformly bounded, we can improve the bound on $\zeta_k$, namely $\zeta_k\in W^{3,p}_{loc}$ for any $p\geq 2$, so that we deduce a uniform $W^{3,p}_{loc}$-bound on $v_k$ as well. 
\end{rmq}

We now prove the nondegeneracy of any minimizer when $\gamma$ is small.
\begin{lem}
\label{strongH2mai}
Assume $0<\sigma<2/N$ and $\gamma_k\to 0$. If $(u_k)_k$ is a sequence of minimizers for  \eqref{MinL2fixed} with $\gamma=\gamma_k$ and $u_k\to w_\mu$ strongly in $H^2(\R^N)$, then $u_k$ is nondegenerate for $k$ large enough. 
\end{lem}

\begin{proof}
We aim at proving that $u_k$ is nondegenerate if $k$ is large enough, i.e. all solutions to the linearization
$$\gamma_k\Delta^2 v- \Delta v +\alpha v= (2\sigma +1)|u_k|^{2\sigma} v$$
are given by $v=\xi\cdot\nabla u_k$, where $\xi \in \R^N$. 
Let us denote the Lagrange multiplier associated to $w_\mu$ by $\alpha_\mu$, namely $w_\mu$ is solution to
\begin{equation}
\label{defalphamupro1}
-\Delta w_\mu +\alpha_\mu w_\mu =|w_\mu|^{2\sigma} w_\mu .
\end{equation}
As recalled at the beginning of the section, $w_\mu$ is nondegenerate. 
We define the bilinear forms  
$$A_k:H^2(\R^N)\times H^2(\R^N)\to \R : (v,w)\mapsto \int_{\R^N}\left(\gamma_k\Delta v\Delta w + \nabla v\nabla w +\alpha_k vw\right)\, dx$$
$$B_k:H^2(\R^N)\times H^2(\R^N)\to \R : (v,w)\mapsto \int_{\R^N}|u_k|^{2\sigma} vw\, dx.$$
The bilinear form $A_k$ defines a scalar product on $H^2(\R^N)$ and since $|u_k|^{2\sigma}\in L^{q/\sigma}(\R^N)$ for every $q\ge 1$, we easily check that the weak convergence induced by $A_k$ implies the strong convergence induced by $B_k$. It then follows from standard arguments that there exist a nondecreasing sequence of positive eigenvalues $(\lambda_{i,k})_i\subset\R^+_0$ and a corresponding orthonormal basis of eigenfunctions $(\phi_{i,k})_i\subset H^2(\R^N)$, namely 
$$A_k(\phi_{i,k},w) = \lambda_{i,k} B_k(\phi_{i,k},w),\quad  \forall\ w\in H^2(\R^N).$$
We choose to normalize the sequence of eigenfunctions with respect to the bilinear form $B_k$.
Since $u_k$ is a minimizer for \eqref{MinL2fixed}, we infer that 
$$A_k(v,v) -(2\sigma+1)B_k(v,v)\ge 0$$
for every $v\in H^2(\R^N)$ satisfying the orthogonality condition $\int_{\R^N}u_kv\, dx =0$. As a simple consequence, $\lambda_0=1$ and $\lambda_i\ge (2\sigma+1)$ for $i\ge 1$. Indeed, $u_k$ is an eigenfunction associated to the eigenvalue $1$ and if there exists another eigenfunction $\psi$ either associated to $1$ or to a distinct eigenvalue smaller than $(2\sigma+1)$, one easily deduces that 
$$A_k(v,v) -(2\sigma+1)B_k(v,v)<0,$$
where
$$v = \psi - \frac{\int_{\R^N}\psi u_k\, dx}{\int_{\R^N}|u_k|^2\, dx}u_k.$$
This is a contradiction because $v$ is orthogonal to $u_k$.
It remains to show that the eigenspace $V_{2\sigma+1}$ associated to $(2\sigma +1)$ is exactly $\{\xi\cdot\nabla u_k: \xi\in\R^N\}$ for large $k$. It is straightforward to check that 
$\{\xi\cdot\nabla u_k: \xi\in\R^N\}\subseteq V_{2\sigma+1}$. We can therefore set
$\phi_{i,k}=\partial_{x_i}u_k$ for $i=1,\ldots,N$. Assume by contradiction that $\phi_{N+1,k}\in V_{2\sigma+1}$. Then, by the variational characterization of eigenvalues, we know that 
$$2\sigma+1= A_k(\phi_{N+1,k},\phi_{N+1,k}) =\min_{v\in H_k}A_k(v,v) ,$$
where $H_k := \{v\in {\rm span}(u_{k},\partial_{x_1}u_k,\ldots,\partial_{x_N}u_k)^\perp: B_k(v,v)=1\}$. It is clear that $\phi_{N+1,k}$ is bounded in $H^1(\R^N)$ 
so that there exists a weak limit $\phi_{N+1}\in H^1(\R^N)$, at least for a subsequence that we still denote by $(\phi_{N+1,k})_k$ for simplicity. We can also assume that $(\phi_{N+1,k})_k$ weakly converges to $\phi_{N+1}$ in $L^{p}(\R^N)$ where $p=1/(1-\sigma)$ if $N\ge 3$ and $p=2/(2-\sigma)$ if $N=1,2$.
Therefore, we have
\begin{multline*}
\int_{\R^N} |\nabla \phi_{N+1}|^2\, dx + \alpha_\mu \int_{\R^N} |\phi_{N+1}|^2 \, dx  \\
\le \liminf_{k\to\infty} \left( \gamma_k\int_{\R^N}|\Delta \phi_{N+1,k}|^2\, dx +\int_{\R^N} |\nabla \phi_{N+1,k}|^2\, dx + \alpha_k \int_{\R^N} |\phi_{N+1,k}|^2 \, dx\right).
\end{multline*}
If $\phi_{N+1}\in {\rm span}(w_{\mu},\partial_{x_1}w_\mu,\ldots,\partial_{x_N}w_\mu)^\perp$ and $\int_{\R^N}|w_\mu|^{2\sigma}\phi_{N+1}^2 \, dx=1$,
 then we have reached a contradiction because the nondegeneracy of $w_\mu$ implies
$$\int_{\R^N} |\nabla \phi_{N+1}|^2\, dx + \alpha_\mu \int_{\R^N} |\phi_{N+1}|^2 \, dx > 2\sigma+1.$$
Therefore, to conclude the proof, it is enough to observe that, as $k\to\infty$, 
$$\int_{\R^N}\phi_{N+1,k}u_{k}\, dx \to \int_{\R^N}\phi_{N+1}w_\mu\, dx,$$
$$\int_{\R^N}\phi_{N+1,k}\partial_{x_i}u_{k}\, dx \to \int_{\R^N}\phi_{N+1}\partial_{x_i}w_\mu\, dx,\quad i=1,\ldots,N$$
and
$$B_k(\phi_{N+1,k},\phi_{N+1,k})\to \int_{\R^N}|w_\mu|^{2\sigma}\phi_{N+1}^2 \, dx.$$
The convergence of the $N+1$ first integrals follow from the strong convergence of $(u_{k})_k$ in $H^1(\R^N)$ and the weak convergence of $(\phi_{N+1,k})_k$. For the last one, we use the weak convergence of $(\phi_{N+1,k})_k$ in $L^p(\R^N)$ and the strong convergence of $(|u_{k}|^{2\sigma})_k$ in the dual space. 

\end{proof}

\begin{thm}
\label{limitfixedmass}
 Assume $0<\sigma<2/N$. Then there exists $\gamma_0>0$ such that if $0<\gamma<\gamma_0$, then \eqref{MinL2fixed} has a unique nondegenerate least energy solution (up to translations). Fixing its maximum at the origin, this solution is radially symmetric. 
\end{thm}
\begin{proof}
Assume that $\gamma_{k}\to 0$ and let $(u_{k})_{k}$ be a sequence such that $ E_{\gamma_{k},1}(u_{k})=I_{\gamma_{k},1}(\mu)$. We prove, by using the previous proposition, that $u_k$ is radially symmetric around some point. Then we use the implicit function theorem to get uniqueness. 

\medbreak 

\noindent\emph{Step one: Radial symmetry of $u_k$.}  We know from Proposition \ref{uniflimit} that, up to a subsequence and a sequence $(y_k)_k\subset\R^N$, we have that $u_{k}(\cdot-y_k)\to w_{\mu}$ strongly in $H^2(\R^N)\cap W^{1,p}(\R^N)$ for every $p\ge 2$. It is enough to prove that $u_k(\cdot-y_k)$ is invariant by rotations around some point. To simplify the notation we just keep $u_k$ to denote $u_k(\cdot-y_k)$. First, observe that $u_k(x)>0$ for $x\in B_R(0)$. Indeed, by Remark \ref{rem:last}, the convergence is locally $C^0$. In particular this shows that for every compact set $\Omega\subset \R^N$, there exists $k_\Omega\in \mathbb N$ such that $u_k(x)>0$ for $x\in \Omega$ and $k\ge k_\Omega$ because $w_{\mu}$ is positive in $\R^N$. In what follows we assume that $k\ge k_{B_R(0)}$. This implies $u_k(x)>0$ for $x\in B_R(0)$. 

\medbreak

Observe now that
$$
V_k^{ij}(x)=  x_i \dfrac{\partial u_k}{\partial x_j}(x)-x_j   \dfrac{\partial u_k}{\partial x_i}(x)
$$
is a solution to $L_k V_k^{ij}(x) =0 $, for any $i\neq j$, where $L_k$ is defined by
$$L_k (v)=\gamma_k \Delta^2 v - \Delta v+\alpha_k v - (2\sigma +1)|u_k|^{2\sigma} v.
$$
On the other hand, by Lemma \ref{strongH2mai}, we know that any solution $\tilde{u}$ to $L_k (\tilde{u})=0$ is of the form $\tilde{u}=\xi\cdot\nabla u_k$, for some $\xi \in \R^N$. 
Hence given $i\ne j\in\{1,\ldots,N\}$, there exists $\xi \in\R^N$ such that 
$$(x_i-\xi_j) \dfrac{\partial u_k}{\partial x_j}(x)-(x_j+\xi_i)  \dfrac{\partial u_k}{\partial x_i}(x) = 
\bar{\xi} \cdot\nabla u_k,
$$
where $\bar{\xi}\in \R^N$ is defined by $\bar\xi_i=\bar\xi_j=0$ and $\bar\xi_\ell = \xi_\ell$ if $\ell\ne i$ and $\ell\ne j$. Obviously $\xi$ depends on $k$. Since $V_k^{ij}=\xi\cdot\nabla u_k$ and $ u_k (x)$ converges locally uniformly in $C^1$ (by Remark \ref{rem:last}), we infer that $V_k^{ij}(x)\to 0$ uniformly in $B_R(0)$. Arguing by contradiction, one then shows that $\xi\to 0$ as $k\to\infty$ otherwise $w_{\mu}$ would be constant in one direction in $B_R(0)$. We may therefore drop the dependence of $\xi$ on $k$ to simplify the notations and assume that $|\xi_i|^2<R/2$ for every $i$. 
We now claim that in fact $\bar \xi =0$. Indeed, if this is not the case, then 
$$\bar \xi \cdot\nabla u_k(x) = 0$$ 
in the subspace $x_i=\xi_j$, $x_j=-\xi_i$. This means $u_k$ is constant, hence zero, in the direction $\bar \xi$ in the subspace $x_i=\xi_j$, $x_j=-\xi_i$ which is a contradiction with the fact that $u_k$ is positive in $B_R(0)$.
Taking polar coordinates $(r,\theta)$ centered at $(\xi_j,-\xi_i)$ in the plane $(x_i,x_j)$, we
conclude that 
$$ \partial_{\theta}u_k(x) =0$$
and therefore $u_k$ is invariant by rotations in the plane $(x_i,x_j)$ around $(\xi_j,-\xi_i)$ fixing all other coordinates.

If $N=2$, we are done as $u_k$ is radially symmetric around $(\xi_j,-\xi_i)$. To get the geometrical intuition, we first assume that $N=3$ before treating the general case. We then infer that $u_k$ is invariant by rotations around three orthogonal axes. By translation invariance of $u_k$, we can assume without loss of generality that the first axe is $A_3=(0,0,x_3)$ and the second one is $A_1=(x_1,b,0)$ for some $b\in\R$. We claim that $b=0$. Indeed, if not, the group action of the rotations around $A_3$ and the rotations around $A_1$ is such that the orbit of any point is unbounded. As $u_k$ tends to zero at infinity, this means $u_k=0$ which is a contradiction. For the same reason, the third axe has to be $A_2=(0,x_2,0)$. Therefore $u_k$ is invariant under the action of the special orthogonal group $SO(3)$. In higher dimensions, $u_k$ is invariant by rotations in $N(N-1)/2$ orthogonal planes, namely rotations in the planes $x_ix_j$ ($i\ne j$) leaving the other coordinates fixed. One then argues similarly to show that there exists one and only one point $\tilde{\xi}\in \R^N$ fixed by all those rotations. Since these rotations generate all rotations in $\R^N$ fixing the point $\tilde{\xi}$ (we recall that the dimension of $SO(N)$ is $N(N-1)/2$), we deduce $u_k$ is radially symmetric around $\tilde \xi$.

%

\medbreak 

\noindent\emph{Step two: Uniqueness.}
In what follows we use the nondegeneracy and radial symmetry of $u_k$ to apply the implicit function theorem. 
 
Let us set $X:=H^2_{rad}(\R^N)$ and $Y:=H^{-2}(\R^N)$ and let $F:\R^{+}\times \R^+ \times X\to Y$ be the operator defined (in the sense of distributions) by 
\begin{equation*}
 F(\gamma,\alpha,u)=\gamma\Delta^2 u-\Delta u+\alpha u-|u|^{2\sigma}u,
\end{equation*}
that is,
\begin{equation*}
F(\gamma,\alpha,u)(v)=
\int_{\R^N}(\gamma\Delta u \Delta v + \nabla u \nabla v +\alpha u v-|u|^{2\sigma}uv)\, dx\hspace{.5cm}\forall\ v\in X.
\end{equation*}
Obviously $F(0,\alpha_\mu,w_\mu)=0$ (see \eqref{defalphamupro1}). Moreover, $F$ is continuously differentiable in a neighbourhood of $(0,\alpha_\mu,w_\mu)$ with $D_u F(\gamma,\alpha,u)\in\mathcal{L}(X,Y)$ defined by
\begin{equation*}
D_u F(\gamma,\alpha,u)v=\gamma \Delta^2 v -\Delta v+\alpha v-(2\sigma+1)|u|^{2\sigma-1}uv,\hspace{.2cm}\forall\ v\in X, 
\end{equation*}
that is,
\begin{equation*}
D_u F(\gamma,\alpha,u)v[w]=\int_{\R^N}(\gamma\Delta v \Delta w + \nabla v \nabla w +\alpha v w-(2\sigma+1)|u|^{2\sigma-1}uvw)\, dx,
\end{equation*}
for all $v,w\in X$. We thus have in the weak sense
\begin{equation*}
 L_{0}(v):=D_u F(0,\alpha_\mu,w_\mu)v=-\Delta v+\alpha_\mu v-(2\sigma+1)w^{2\sigma}_\mu v.
\end{equation*}

It is well-known that the kernel of $L_{0}$ is of dimension $N$ when considered in $H^2(\R^N)$ and it is spanned by the partial derivatives of $w_{\mu}$. In particular, the kernel of $L_{0}$ restricted to $H^2_{rad}(\R^N)$ is trivial and $L_{0}:X\to Y$ is one-to-one, see for instance \cite{ChGuNaTs,KaTa,Kwong}. Moreover, it follows from the open mapping theorem that $L_{0}^{-1}:Y\to X$ is continuous. 

Since the linear map $L_{0}$ is a homeomorphism, we can apply the implicit function theorem. Namely, there exist $\gamma_0, \varepsilon>0$ and an open set $U_0\subset X$ that contains $w_\mu$ such that for every $\gamma\in [0,\gamma_0[$ and $\alpha \in (\alpha_\mu-\varepsilon , \alpha_\mu+\varepsilon )$, the equation $F(\gamma,\alpha,u)=0$ has a unique solution $u_{\gamma,\alpha}\in U_0$ and the curve 
$$\Gamma:[0,\gamma_0[\, \times (\alpha_\mu - \varepsilon ,\alpha_\mu +\varepsilon ) \to H^2(\R^N) : (\gamma , \alpha) \mapsto u_{\gamma,\alpha}$$ 
is of class $C^{1}$.

Now suppose that the uniqueness of least energy solutions fails in every set of the form $(0,\gamma)\times (\alpha_\mu - \varepsilon ,\alpha_\mu+\varepsilon ) $. We can then construct two sequences in $M_{\mu}$ of least energy solutions along a sequence $\gamma_k$ converging to $0$ and a sequence $\alpha_k \rightarrow \alpha_0$. We call them $(u_k)_k$ and $(v_k)_k$ and denote by $I_{\gamma_k,1}({\mu})$ their common energy. By assumption, $u_k\ne v_k$. By the previous step, we know that $u_k$ and $v_k$ are radially symmetric. 
Since these two sequences converge in $H^2$ to $u_0$ as $k\to\infty$, we have 
$$u_k,v_k\to u_0,$$
where the convergence is strong in $H^2$. 
Then, for $k$ large enough, there exist two solutions to the equation $F(\gamma_{k},\alpha_k,u)=0$ in $U_0$ with $\gamma_{k}<\gamma_{0}$ and $\alpha_k \in  (\alpha_\mu - \varepsilon ,\alpha_\mu+\varepsilon )$. This is a contradiction and ends the proof. 
\end{proof}


\subsection{Ground states}
It is also possible to obtain the same kind of result for solutions to 
\begin{equation}
\label{min1gamma}
m=\inf_{u\in M}J_{\gamma,\beta,\alpha}(u),
\end{equation}
where $J_{\gamma,\beta,\alpha}(u)$ is defined in \eqref{J-quadr} and $M$ in \eqref{lpconsint}. The proof is very similar and easier than the one above. We refer to  \cite[Propositions $1$ \& $2$]{bonnas} for more details. Notice that we are able to improve \cite[Proposition 2]{bonnas} proceeding as in Proposition \ref{uniflimit}. 
\begin{thm}
\label{limitlpfixed}
 Assume that $0<\sigma<2/N$ and $\alpha,\beta >0$. Then there exists $\gamma_0>0$ such that if $0<\gamma<\gamma_0$, then \eqref{min1gamma} has a unique nondegenerate positive least energy solution (up to translations). Fixing its maximum at the origin, this solution is radially symmetric and radially decreasing.  
\end{thm}
Thanks to the previous theorem, we are able to obtain the positivity and strict decay of the unique solution to \eqref{MinL2fixed} when $\gamma$ is small enough. We have seen at the end of the proof of Theorem \ref{limitfixedmass} that there exist $\gamma_0, \varepsilon>0$ and an open set $U_0\subset H^2_{rad}$ that contains $w_\mu$ such that for every $\gamma\in [0,\gamma_0[$ and $\alpha \in (\alpha_\mu-\varepsilon , \alpha_\mu+\varepsilon )$, the equation $F(\gamma,\alpha,u)=0$ has a unique solution $u_{\gamma,\alpha}\in U_0$. Therefore, we see that when $\gamma$ is small enough, the solution given in Theorem \ref{limitlpfixed} when $\alpha=\alpha_\mu$ coincides with the solution to Theorem \ref{limitfixedmass}. We deduce the following corollary.

\begin{cor}
Assume that $0<\sigma<2/N$ and $\beta>0$. Then there exists $\gamma_0>0$ such that if $0<\gamma<\gamma_0$, then, up to its sign, the unique solution to \eqref{MinL2fixed} is positive and radially decreasing.
\end{cor}

\section{Orbital stability}\label{sec:orbstab}


In this section we establish some results concerning the orbital stability of standing wave solutions to the following equation
\begin{equation}
\label{complexsch}
i\partial_{t}\psi -\Delta^2 \psi +\beta \Delta \psi +|\psi|^{2\sigma} \psi=0.
\end{equation}
First, let us introduce some notation and some basic facts. It is easy to see that the energy 
\begin{equation}\label{Orbital-Energy}
E_{1,\beta}(\psi)= \dfrac{1}{2} \int_{\R^N} |\Delta \psi|^2\, dx +\dfrac{\beta}{2}\int_{\R^N}|\nabla \psi|^2\, dx- \dfrac{1}{2\sigma +2}\int_{\R^N}|\psi|^{2\sigma +2}\, dx,
\end{equation}
and the mass
\begin{equation}\label{Orbital-Mass}
F(\psi)=\dfrac{1}{2}\int_{\R^N} |\psi|^2\, dx
\end{equation}
are conserved by solutions to \eqref{complexsch}. 
Let $\Lc^2= L^2 (\R^N) \times L^2 (\R^N)$ and $\Hc^s= H^s (\R^N)\times H^s (\R^N)$ with $s\in\N$.
Writing $\psi=u_1+i u_2$, we see that \eqref{complexsch} gives rise to the system
$$\left\{\begin{array}{rl}  \partial_{t}u_1 -\Delta^2 u_2 +\beta \Delta u_2 +(|u_1|^2+|u_2|^2)^{\sigma} u_2 &=0\\
-\partial_{t}u_2 -\Delta^2 u_1 +\beta \Delta u_1 +(|u_1|^2+|u_2|^2)^{\sigma} u_1&=0.
\end{array}
\right.$$
The energy and the mass are then given by
$$E_{1,\beta}(\psi)=\dfrac{1}{2}\int_{\R^N}( |\Delta u_1|^2+ |\Delta u_2|^2)\, dx +\dfrac{\beta}{2}\int_{\R^N}(|\nabla u_1|^2 + |\nabla u_2|^2)\, dx$$ 
$$- \dfrac{1}{2\sigma +2}\int_{\R^N}(|u_1|^2+|u_2|^2)^{\sigma +1}\, dx,$$
and
$$F(\psi )=\dfrac{1}{2}\int_{\R^N} (|u_1|^2+|u_2|^2)\, dx.$$
So we can write \eqref{complexsch} as
\begin{equation}
\label{eqmat}
\dfrac{d}{dt}\Phi(t)= J E_{1,\beta}^\prime (\Phi (t)),\hspace{.3cm} \Phi=\begin{pmatrix}u_1\\ u_2\end{pmatrix},
\end{equation}
where $E_{1,\beta}^\prime$ is the Fr\' echet derivative of $E_{1,\beta}$ with respect to $\Phi$, and 
$$J=\begin{pmatrix}0 & -1\\ 1& 0\end{pmatrix}.$$
We recall the local and global wellposedness properties of this system.
\begin{thm}[{\cite[Proposition 4.1 \& Corollary 4.1]{MR2353631}}]
Let $0<\sigma <4/N$. Given $\Phi_0 \in \Hc^2$, there exists a solution $U\in C(\mathbb{R},\Hc^2)$ of \eqref{eqmat} such that $\Phi(0)=\Phi_0$. 
Moreover, for all $t\geq 0$, the map $\Phi_0\mapsto\Phi$ is continuous from $\Hc^2$ to $C([0,t],\Hc^2)$.
\end{thm} 
Notice that \eqref{complexsch} is invariant under the unitary action of rotations and translations, i.e. if $\Phi=\begin{pmatrix}u_1\\ u_2\end{pmatrix}$ is a solution to \eqref{eqmat}, then 
\begin{equation*}
T_1(\theta)\Phi=\begin{pmatrix}\cos \theta & \sin \theta \\ -\sin \theta & \cos \theta\end{pmatrix} \begin{pmatrix}u_1\\ u_2\end{pmatrix}\hspace{.5cm}\text{and}\hspace{.5cm}T_2(r) \Phi=\begin{pmatrix}u_1(\cdot-r,\cdot)\\ u_2(\cdot-r,\cdot)\end{pmatrix} 
\end{equation*}
are also solutions. Let $e^{i \alpha t} U$ be a standing wave solution to \eqref{complexsch}. 
Denote $\U=(U,0)$. We define the orbit generated by $\U$ as
\begin{align*}
\Omega_\U &:=\Big\{T_1 (\theta) T_2 (r) \U: \theta ,r \in \R\Big\}\\
&\approx \left\{\begin{pmatrix}\cos \theta & \sin \theta \\ -\sin \theta & \cos \theta\end{pmatrix}\begin{pmatrix}U(\cdot-r)\\ 0 \end{pmatrix}: \theta ,r \in \R\right\}.
\end{align*}
For any $f,g\in \Hc^2$, we define 
$$d(f,g):=\inf \Big\{ \left\|f-T_1(\theta)T_2(r)g\right\|_{\Hc^2}: \theta,r \in \R\Big\}.$$
Notice that $d(f,\U)=d(f,\Omega_\U)$.
\begin{defi}
\label{defiorb}
Let $\Theta(x,t)=e^{i \alpha t}U(x)$ be a standing wave of \eqref{complexsch}. We say that $\Theta$ is orbitally stable in $\Hc^2$ if, given $\varepsilon>0$, there exists $\delta>0$ such that if $\Phi_0\in \Hc^2$ satisfies $\left\|\Phi_0-\U\right\|_{\Hc^2}<\delta$, then the solution $\Phi(t)$ of \eqref{complexsch} with initial data $\Phi_0$ exists for all $t\geq 0$ and satisfies 
$$d(\Phi (t),\Omega_\U)<\varepsilon,\ for\ all\ t\geq 0.$$
\end{defi}
We also define a weaker notion of stability.
\begin{defi}
\label{defiGorb}
Let $\mathcal{G}:= \{U\in  H^2 (\R^N):\ U\ \text{is a solution to} \ \eqref{MinL2fixed} \}$. We say that the set $\mathcal{G}$ is stable in $H^2$ if, given $\varepsilon>0$, there exists $\delta>0$ such that if $\Phi_0\in \Hc^2$ satisfies $\left\|\Phi_0-\U\right\|_{\Hc^2}<\delta$ for some $U\in \mathcal{G}$, then the solution $\Phi(t)$ to \eqref{complexsch} with initial data $\Phi_0$ exists for all $t\geq 0$ and satisfies 
$$d(\Phi (t),\Omega_V)<\varepsilon,\ for\ all\ t\geq 0,$$
for some $V\in \mathcal{G}$.
\end{defi}

We prove two results in this section. First, as a direct consequence of Theorem \ref{Compact-Min-Sol}, using the method of Cazenave and Lions \cite{Caz}, we prove the stability of the set $\mathcal{G}$. As already pointed out, see page \pageref{page-C-L}, this method does not apply to ground state of the functional energy. Then, we establish the orbital stability of nondegenerate (see \eqref{4nlsgamma-beta-alpha-linearized}) ground state solutions satisfying \eqref{extracond}, i.e. solutions to \eqref{MinLpintro}. 

\subsection{Standing waves with a prescribed mass}
In this subsection we show the stability of the set of solutions obtained in Theorem \ref{Compact-Min-Sol}.
\begin{thm}
The set $\mathcal{G}$ is stable. 
\end{thm}
\begin{proof}
Let $U\in \mathcal{G}$. Assume by contradiction that there exists a sequence of solutions $(u_k)_k$ of \eqref{complexsch} with $u_k (0)=\varphi_k$, for $(\varphi_k)_k \subset H^2 (\R^N)$ such that $\lim_{k\rightarrow \infty}\left\|\varphi_k -U\right\|_{H^2}=0 $ and such that there exists $(t_k)_k\subset\R^{+}$ with $d(u_k(t_k), \Omega_U)\geq \varepsilon$, for some $\varepsilon>0$ fixed. Using the conservation of the energy and the mass, it is easy to see that $\left\|u_k (t_k)\right\|_{L^2}\rightarrow \left\|U\right\|_{L^2}$ and $E_{1,\beta} (u_k (t_k))\rightarrow E_{1,\beta} (U)=I_{1,\beta}$ as $k\to\infty$. Therefore, using Theorem \ref{Compact-Min-Sol}, we get that $d(u_k(t_k), \Omega_V)\rightarrow 0$, for some $V\in \mathcal{G}$ which gives a contradiction.
\end{proof}
\begin{rmq}
Observe that the previous theorem implies that the sign-changing solution, which existence is assured by Corollary \ref{sgnchmassfix} belongs to a stable set. 
\end{rmq}
\subsection{Nondegenerate standing waves}
In our last subsection we investigate the orbital stability of nondegenerate ground state solutions $U\in H^2 (\R^N)$ to \eqref{eq4beta}. Let us recall (see Theorem \ref{Radial-LES}) that  solutions to \eqref{min1} are real up to a phase.
Before proceeding, we introduce some notations.
Denote $G=E_{1,\beta}+\alpha F$, where $E_{1,\beta}$ and $F$ are given in \eqref{Orbital-Energy} and \eqref{Orbital-Mass}, respectively. Let $\U=(U,0)$, where $U$ is a solution to \eqref{min1}. It is easy to see that 
$$G^\prime ( \U)=E_{1,\beta}^\prime (\U)+\alpha F^\prime (\U)=0.
$$
We set
\begin{align}\label{L-operateur-Orb}
\Lb = \begin{pmatrix}L_1 & 0\\ 0 & L_2\end{pmatrix},
\end{align}
where $L_1$ and $L_2$ denote the linearization of \eqref{eqmat} around the standing wave $U$, i.e.
$$L_1 (u)=\Delta^2 u -\beta \Delta u +\alpha u -(2\sigma +1) |U|^{2\sigma} u,$$
and
$$L_2 (u)=\Delta^2 u -\beta \Delta u +\alpha u - |U|^{2\sigma} u.$$
We assume that $U$ is nondegenerate, that is, all solutions to $L_1 (u)=0$ are of the form $u=\xi\cdot\nabla U$, for any $\xi \in \R^N$. 
We know that $(L_1 u,u)_{L^2}\geq 0$ for any $u\in H^2$ such that $( U, u )_{H^2}=0$ since $U$ is a ground state. 
This implies that $(L_2 u,u)_{L^2}> 0$ for all $(u,U)_{H^2}=0$. As $L_2 U=0$, this implies that the first eigenvalue is zero and that it is simple. 
Therefore, we have the following result.
\begin{prop}
Assume that $U$ is a nondegenerate ground state. Then the operator $\Lb$ given in \eqref{L-operateur-Orb} and defined on $\Lc^2$ with domain $\Hc^4$ has a unique simple negative eigenvalue. The eigenvalue $0$ is of multiplicity $N+1$ with associated eigenfunctions $(e_i\cdot\nabla U,0)$, $i=1,\ldots ,N$ and $(0,U)$.
\end{prop}

Next, we prove that any nondegenerate ground state satisfying \eqref{extracond} is orbitally stable.

\begin{thm}
\label{thmorbsta}
Let $U$ be a solution to \eqref{eq4beta} and assume it is nondegenerate and \eqref{extracond} holds. Then, the standing wave $\Theta(x,t)=e^{i\alpha t}U(x)$ is orbitally stable in $\Hc^2$.
\end{thm}
As an immediate consequence of Theorem \ref{thmnondegen}, we obtain the following result.
\begin{cor}
Assume that $N=1$. Let $U$ be a solution to \eqref{eq4beta} satisfying \eqref{extracond}.  Then, the standing wave $\Theta(x,t)=e^{i\alpha t}U(x)$ is orbitally stable in $\Hc^2$.
\end{cor}
Given $\varepsilon>0$, we define $\Omega^\varepsilon_\U=\{v\in \Hc^2 : d(v,\Omega_\U)<\varepsilon\}$. In order to prove the previous theorem, we shall construct a Lyapunov function. Once we have this function, the proof of Theorem \ref{thmorbsta} will then follow from \cite[Proposition $4.1$]{Stuart}. Before proceeding let us recall the following definition.
\begin{defi}
\label{deflyap}
A function $V:\Hc^2 \rightarrow \R$ is a Lyapunov function for the orbit of $\Omega_{\U}$ if the following properties hold:
\begin{enumerate}
\item There exists $\rho>0$ such that $V:\Omega_{\U}^\rho \rightarrow \R$ is $C^2$ and, for all $v\in \Omega_\U$,
$$V(v)=0\ and\ V^\prime (v)=0.$$
\item For all $v\in \Omega_{\U}^\rho$, there holds
$$\left\langle V^\prime (v),Jv\right\rangle =\left\langle V^\prime (v),x_i\cdot\nabla v\right\rangle=0,\ i=1,\ldots ,N.$$
\item If $\Phi(t)$ is a global solution to the Cauchy problem \eqref{eqmat} with initial data $\Phi_0$, then $V(\Phi (t))=V(\Phi_0)$, for all $t\geq 0$.
\item There exists $c>0$ such that, for all $v\in \Omega_\Phi^\rho$,
$$V(v)\geq c[d(v,\Omega_\U)]^2.$$
\end{enumerate}
\end{defi}
\begin{prop}\label{prop-Lyap-3}
There exists $C>0$ such that the functional $V:\Hc^2 \rightarrow \R$ defined by
$$V(v)=G(v)-G(\U)+C (F(v)-F(\U))^2,$$
is a Lyapunov function for the orbit $\Omega_\U$.
\end{prop}
In order to prove this proposition, we follow very closely \cite{natalipastor}. Observe that it is straightforward to check that $V$ satisfies the points $(1)-(3)$ in the Definition \ref{deflyap}. To prove that $V$ satisfies the last point, we need the next two lemmas.

\begin{lem}
\label{lyaplem1}
There exists $R>0$, depending on $\U$, such that for all $\rho \in (0,R)$ and $v\in \Omega_\U^\rho$, there exist $r_1,\theta_1\in \R$ such that
$$\left\|v-T_1(\theta_1)T_2(r_1)\U\right\|_{\Hc^2}<\rho,$$
and
\begin{align*}
&\Big(v-T_1(\theta_1)T_2(r_1)\U, J T_1(\theta_1) T_2 (r_1) \U\Big)_{\Hc^2}\\
&=\Big(v-T_1(\theta_1)T_2(r_1)\U,  T_1(\theta_1) T_2 (r_1) \xi\cdot\nabla\U\Big)_{\Hc^2}=0.
\end{align*}
\end{lem}
\begin{proof}
See \cite[Lemma $4.13$]{natalipastor}.
\end{proof}
Before stating the next lemma, let us introduce some notation. Let $G^{\prime \prime}$ be the second order Fr\' echet derivative of $G$, i.e. $G^{\prime \prime}(\U)v=\Ic \Lb v $, for $v\in \Hc^6$, where $\Ic :\Hc^2 \rightarrow \Hc^{-2}$ is the natural injection with respect to the inner product of $\Lc^2$, i.e.
$$\left\langle \Ic u,v\right\rangle =(u,v)_{\Lc^2},\hspace{.3cm}\text{for all}\ u,v\in \Hc^2.$$
Let $S:\Hc^2 \rightarrow \Hc^2$ be the self-adjoint operator defined by $S=\Rc^{-1}G^{\prime \prime}(\U)$ where $\Rc : \Hc^2\rightarrow \Hc^{-2}$ is the Riesz isomorphism with respect to the inner product of $\Hc^2$, that is
$$\left\langle \Rc u,v\right\rangle =(u,v)_{\Hc^2},\hspace{.3cm}\text{for all}\ u,v\in \Hc^2.$$

\begin{lem}
\label{lyaplem2}
There exist two positive constants $C$ and $\delta$ such that
$$(Sv,v)_{\Hc^2}+2C (\Rc^{-1}\Ic \U,v)^2_{\Hc^2}\geq \delta \left\|v\right\|_{\Hc^2}, $$
for all $v\in \{J\U , \xi\cdot\nabla \U\}^\perp=\big\{v\in\Hc^2:(J\U,v)_{\Hc^2}=(\xi\cdot\nabla \U,v)_{\Hc^2}=0\big\}$.
\end{lem}
\begin{proof}
Following the proof of Lemmas $4.12$ and $4.15$ from \cite{natalipastor}, we observe that the proof of Lemma \ref{lyaplem2} reduces to the following claim.

\medbreak

\noindent\emph{Claim: Let $v=(P,Q)\in \Hc^4$ be such that
$$(Q,U)_{L^2}=(P,U)_{L^2}=(P,\xi\cdot\nabla U)_{L^2}=0,$$
then there exists $\delta>0$ such that
\begin{equation}\label{mainestlyaplem}
(\Lb v,v )_{\Lc^2}\geq \delta \left\|v\right\|_{\Lc^2}^2.
\end{equation}
}

\noindent We split the proof of \eqref{mainestlyaplem} into two steps.

\medbreak

\noindent\emph{Step one: We have
$$\inf_{\substack{(v,U)_{L^2}=0\\ v\neq 0}} \dfrac{(L_2 v,v)_{L^2}}{ \left\|v\right\|_{L^2}^2}>0.$$
}

\noindent Seeking a contradiction, we assume that this infimum is zero. 
Let $(v_k)_k$ be a $L^2$-normalized minimizing sequence. By using the exponential decay of $U$  it is easy to show that $(v_k)_k$ is bounded in $H^2(\R^N)$. Therefore, there exists a function $v$ such that $v_k\rightharpoonup v$ weakly in $H^2$ and $(v,U)_{L^2}=0$. 
Using once again the exponential decay of $U$, one can prove that $(v_k)_k$ converges strongly in $H^2$ to $v$ 
and that $v$ solves the minimization problem. Thus, there exist $\theta_1$, $\theta_2\in \R$ such that for all $\psi\in H^2$,
$$(L_2 v,\psi)_{L^2}= \theta_1 (v,\psi)_{L^2}+\theta_2 (U,\psi)_{L^2}.$$
Taking $\psi=v$, we get $\theta_1=0$. Next, if we take $\psi =U$, we get that $\theta_2=0$ so $v=cU$, for any constant $c$, since $L_2$ is nondegenerate. This implies that $\omega>0$, contradicting $\omega=0$.



\medbreak

\noindent\emph{Step two: We have
$$\delta_1=\inf \Big\{(L_1 u,u)_{L^2}: \left\|u\right\|_{L^2}=1,\ (u,U)_{L^2}=(u,\xi\cdot\nabla U)_{L^2}=0 \Big\}>0.$$
}

\noindent First, we show that
$$\gamma=\inf \Big\{ (L_1 u,u)_{L^2} : u\in H^2,\ \left\|u\right\|_{L^2}=1,\ (u,U)_{L^2}=0\Big\}=0.$$
Indeed, since $L_1 (\xi\cdot\nabla U)=0$, we have $\gamma\leq 0$. We notice that $(L_1 U,U)_{L^2}\leq 0$. 
Using \cite[Proposition 4.3]{MR820338}, this implies that $\gamma\geq 0$. We deduce from this that $\delta_1\geq 0$. Assume now that the infimum is achieved by a function $\varphi$. Then there exist $m_1,m_2,m_{3,1},\ldots , m_{3,N} \in \R$ such that
$$L_1 \varphi= m_1 \varphi+m_2 U+\sum_{i=1}^N m_{3,i} x_i\cdot\nabla U.$$
Assume by contradiction that $\delta_1=0$. Consequently, we have that $m_1=0$. Testing the equation with $\xi\cdot\nabla U\in \ker L_1$, we deduce that $m_{3,i}=0$, for any $i=1,\ldots ,N$. Next, we show that $m_2\neq 0$.  By contradiction, assume that $m_2=0$. Then, using the nondegeneracy of $L_1$, we deduce that $\varphi=C\tilde{\xi}\cdot\nabla U$, for some constant $C$ and some $\tilde{\xi}\in \R^N$. On the other hand, by definition we have $(\varphi,\tilde{\xi}\cdot\nabla U)_{L^2}=0$. This implies that $C=0$, contradicting $\left\|\varphi\right\|_{L^2}=1$. 
Therefore, we deduce that $\varphi$ is a solution to $L_1 \varphi = m_2 U$, for some $m_2\neq 0$, and $0=(\varphi,U)_{L^2}$,  
which yields a contradiction with \eqref{extracond}.

\noindent Finally, we are now in position to prove \eqref{mainestlyaplem}. Thanks to the two previous steps, there exist $\delta_1,\delta_2>0$ such that $(L_2 Q,Q)_{L^2} \geq \delta_2 \left\|Q\right\|^2_{L^2}$ and $(L_1 P,P)_{L^2} \geq \delta_1 \left\|P\right\|^2_{L^2}$. Recalling that $\Lb v =(L_1 P, L_2 Q)$, we get \eqref{mainestlyaplem} with $\delta=\min \{\delta_1 ,\delta_2\}$.

\end{proof}

\noindent Now we give the proof of Proposition \ref{prop-Lyap-3}.

\begin{proof}[Proof of Proposition \ref{prop-Lyap-3}.] Using Lemmas \ref{lyaplem1} and \ref{lyaplem2} we show that there exists $c>0$ such that, for all $v\in \Omega_\Phi^\rho$,
$$V(v)\geq c[d(v,\Omega_\U)]^2.$$
Since $V$ is $C^2$ and noticing that $V(\U)=0$ and $V^\prime (\U)=0$, a Taylor expansion gives
$$V(v)=\dfrac{1}{2} \big\langle V^{\prime \prime} (\U) (v-\U),v-\U \big\rangle + h(v),$$
where $h$ is a function satisfying
$$\lim_{v\rightarrow \U} \dfrac{h(v)}{\left\|v-\U\right\|^2_{\Hc^2}}=0.$$
Let $R$ be the constant appearing in Lemma \ref{lyaplem1} and $\delta$ be that appearing in Lemma \ref{lyaplem2}. We can assume that, for some $\rho \in (0,R/2)$,
$$|h(v)|\leq \dfrac{\delta}{4} \left\|v-\U\right\|^2_{\Hc^2},\ \text{for all}\ v\in B_\rho (\U).$$
Let $v\in \Omega_{\U}^\rho$. We deduce from Lemma \ref{lyaplem1}, that there exist $\theta_1$ and $r_1$ such that $u=T_1 (-\theta_1) T_2 (-r_1)v$ satisfies $\left\|u-\U\right\|_{\Hc^2}<\rho$ and $u-\U \in  \{J\U ,\xi\cdot\nabla \U\}^\perp$. On the other hand, using Lemma \ref{lyaplem2}, there exist $\delta$ and $C$ such that, for all $w\in \{J\U , \xi\cdot\nabla \U\}^\perp$, we have
\begin{align*}
\left\langle V^{\prime \prime} (\U) w,w\right\rangle &= \left\langle G^{\prime \prime} (\U)w,w\right\rangle +2C \left\langle F^\prime (\U),w\right\rangle^2 \\
&=\left(S w,w\right)_{\Hc^2}+2C \left( \Rc^{-1} \Ic \U , w\right)^2_{\Hc^2}\\
&\geq  \delta \left\|v\right\|_{\Hc^2} .
\end{align*}
Combining the previous estimates, we deduce that, for any $v\in \Omega_{\U}^\rho$,
\begin{align*}
V(v)=V(u) &\geq \delta \left\|u-\U\right\|_{\Hc^2} - \dfrac{\delta}{4} \left\|u-\U\right\|^2_{\Hc^2}\\
& \geq  \dfrac{\delta}{4} d(u,\Omega_\U)^2 = \dfrac{\delta}{4} d(v,\Omega_\U)^2. 
\end{align*}
Therefore, $V$ is a Lyapunov function. This concludes the proof.
\end{proof}

\begin{proof}[Proof of Theorem \ref{thmorbsta}.] In view of Proposition \ref{prop-Lyap-3}, all the assumptions of \cite[Proposition $4.1$]{Stuart} are satisfied and hence the conclusion follows. 
\end{proof}





\bibliographystyle{plain}
\bibliography{biblio}

\begin{thebibliography}{10}

\bibitem{MR1151253}
J.~P. Albert.
\newblock Positivity properties and stability of solitary-wave solutions of
  model equations for long waves.
\newblock {\em Comm. Partial Differential Equations}, \textbf{17}(1-2):1--22,
  1992.

\bibitem{alvino1999}
A.~Alvino, G.~Trombetti, P.-L. Lions, and S.~Matarasso.
\newblock Comparison results for solutions of elliptic problems via
  symmetrization.
\newblock {\em Ann. Inst. H. Poincar{\'e} Anal. Non Lin{\'e}aire},
  \textbf{16}(2):167--188, 1999.

\bibitem{amick}
C.J. Amick and J.F. Toland.
\newblock Global uniqueness of homoclinic orbits for a class of fourth order
  equations.
\newblock {\em Z. Angew. Math. Phys.}, \textbf{43}(4):591--597, 1992.

\bibitem{MR1745182}
M.~Ben-Artzi, H.~Koch, and J.-C Saut.
\newblock Dispersion estimates for fourth order {S}chr\"odinger equations.
\newblock {\em C. R. Acad. Sci. Paris S\'er. I Math.}, \textbf{330}(2):87--92,
  2000.

\bibitem{MR695535}
H.~Berestycki and P.-L. Lions.
\newblock Nonlinear scalar field equations. {I}. {E}xistence of a ground state.
\newblock {\em Arch. Rational Mech. Anal.}, \textbf{82}(4):313--345, 1983.

\bibitem{MR695536}
H.~Berestycki and P.-L. Lions.
\newblock Nonlinear scalar field equations. {II}. {E}xistence of infinitely
  many solutions.
\newblock {\em Arch. Rational Mech. Anal.}, \textbf{82}(4):347--375, 1983.

\bibitem{BCGJ}
D.~Bonheure, J.-B. Casteras, T.~Gou, and L.~Jeanjean.
\newblock Normalized solutions to the mixed dispersion nonlinear
  {S}chr\"odinger equation in the mass critical and supercritical regime.
\newblock {\em preprint}, 2017.

\bibitem{BCGJ1}
D.~Bonheure, J.-B. Casteras, T.~Gou, and L.~Jeanjean.
\newblock Strong instability of ground states to a fourth order {S}chr\"odinger
  equation.
\newblock {\em To appear in IMRN, doi:10.1093/imrn/rnx273}, 2017.

\bibitem{BoMSRa}
D.~Bonheure, E.~Moreira~dos Santos, and M.~Ramos.
\newblock Ground state and non-ground state solutions of some strongly coupled
  elliptic systems.
\newblock {\em Trans. Amer. Math. Soc.}, \textbf{364}(1):447--491, 2012.

\bibitem{bonnas}
D.~Bonheure and R.~Nascimento.
\newblock Waveguide solutions for a nonlinear {S}chr\"odinger equation with
  mixed dispersion.
\newblock {\em Contributions to Nonlinear Elliptic Equations and Systems,
  Progr. in Nonlinear Differential Equations and Appl.}, \textbf{86}:31--53,
  2015.

\bibitem{boulengerlenzmann}
Thomas Boulenger and Enno Lenzmann.
\newblock Blowup for biharmonic {NLS}.
\newblock {\em Ann. Sci. \'Ec. Norm. Sup\'er. (4)}, 50(3):503--544, 2017.

\bibitem{BrZi}
J.E. Brothers and W.P. Ziemer.
\newblock Minimal rearrangements of {S}obolev functions.
\newblock {\em J. Reine Angew. Math.}, \textbf{384}:153--179, 1988.

\bibitem{Buffoni1995109}
B.~Buffoni.
\newblock Infinitely many large amplitude homoclinic orbits for a class of
  autonomous hamiltonian systems.
\newblock {\em Journal of Differential Equations}, \textbf{121}(1):109 -- 120,
  1995.

\bibitem{Buffoni1996}
B.~Buffoni.
\newblock Periodic and homoclinic orbits for lorentz-lagrangian systems via
  variational methods.
\newblock {\em Nonlinear Analysis}, \textbf{26}(3):443--462, 1996.

\bibitem{buffonitoland}
B.~Buffoni, A.R. Champneys, and J.F. Toland.
\newblock Bifurcation and coalescence of a plethora of homoclinic orbits for a
  {H}amiltonian system.
\newblock {\em J. Dynam. Differential Equations}, \textbf{8}(2):221--279, 1996.

\bibitem{MR1755067}
J.~Busca and B.~Sirakov.
\newblock Symmetry results for semilinear elliptic systems in the whole space.
\newblock {\em J. Differential Equations}, \textbf{163}(1):41--56, 2000.

\bibitem{dolb}
I.~Catto, J.~Dolbeault, O.~S{\'a}nchez, and J.~Soler.
\newblock Existence of steady states for the
  {M}axwell-{S}chr\"odinger-{P}oisson system: exploring the applicability of
  the concentration-compactness principle.
\newblock {\em Math. Models Methods Appl. Sci.}, \textbf{23}(10):1915--1938,
  2013.

\bibitem{Caz}
T.~Cazenave and P.-L. Lions.
\newblock Orbital stability of standing waves for some nonlinear
  {S}chr\"odinger equations.
\newblock {\em Comm. Math. Phys.}, \textbf{85}(4):549--561, 1982.

\bibitem{ChGuNaTs}
S.M. Chang, S.~Gustafson, K.~Nakanishi, and T.P. Tsai.
\newblock Spectra of linearized operators for {NLS} solitary waves.
\newblock {\em SIAM J. Math. Anal.}, \textbf{39}(4):1070--1111, 2008.

\bibitem{MR2361740}
Y.~Deng and Y~Li.
\newblock Exponential decay of the solutions for nonlinear biharmonic
  equations.
\newblock {\em Commun. Contemp. Math.}, \textbf{9}(5):753--768, 2007.

\bibitem{MR0174304}
G.~H. Derrick.
\newblock Comments on nonlinear wave equations as models for elementary
  particles.
\newblock {\em J. Mathematical Phys.}, \textbf{5}:1252--1254, 1964.

\bibitem{MR1898529}
G.~Fibich, B.~Ilan, and G.~Papanicolaou.
\newblock Self-focusing with fourth-order dispersion.
\newblock {\em SIAM J. Appl. Math.}, \textbf{62}(4):1437--1462 (electronic),
  2002.

\bibitem{deFig-MS-Mi}
D.~G.~de Figueiredo, E.~Moreira dos Santos, and O.~H. Miyagaki.
\newblock Sobolev spaces of symmetric functions and applications.
\newblock {\em J. Funct. Anal.}, \textbf{261}(12):3735--3770, 2011.

\bibitem{Ga-57}
E.~Gagliardo.
\newblock Caratterizzazioni delle tracce sulla frontiera relative ad alcune
  classi di funzioni in {$n$} variabili.
\newblock {\em Rend. Sem. Mat. Univ. Padova}, \textbf{27}:284--305, 1957.

\bibitem{Ga-58}
E.~Gagliardo.
\newblock Propriet\`a di alcune classi di funzioni in pi\`u variabili.
\newblock {\em Ricerche Mat.}, \textbf{7}:102--137, 1958.

\bibitem{MR3462127}
Q.~Guo.
\newblock Scattering for the focusing {$L^2$}-supercritical and
  {$\dot{H}^2$}-subcritical biharmonic {NLS} equations.
\newblock {\em Comm. Partial Differential Equations}, \textbf{41}(2):185--207,
  2016.

\bibitem{HOMBURG2010379}
Ale~Jan Homburg and Bjorn Sandstede.
\newblock Chapter 8 - homoclinic and heteroclinic bifurcations in vector
  fields.
\newblock volume~3 of {\em Handbook of Dynamical Systems}, pages 379 -- 524.
  Elsevier Science, 2010.

\bibitem{jeanjean}
L.~Jeanjean and T.~Luo.
\newblock Sharp nonexistence results of prescribed {$L^2$}-norm solutions for
  some class of {S}chr\"odinger-{P}oisson and quasi-linear equations.
\newblock {\em Z. Angew. Math. Phys.}, \textbf{64}(4):937--954, 2013.

\bibitem{KaTa}
Y.~Kabeya and K.~Tanaka.
\newblock Uniqueness of positive radial solutions of semilinear elliptic
  equations in {$\mathbb{R}^N$} and {S}\'er\'e's non-degeneracy condition.
\newblock {\em Comm. Partial Differential Equations},
  \textbf{24}(3-4):563--598, 1999.

\bibitem{MR1779828}
V.I. Karpman and A.G. Shagalov.
\newblock Stability of solitons described by nonlinear {S}chr\"odinger-type
  equations with higher-order dispersion.
\newblock {\em Phys. D}, \textbf{144}(1-2):194--210, 2000.

\bibitem{kry}
N.V. Krylov.
\newblock {\em Lectures on elliptic and parabolic equations in {S}obolev
  spaces}, volume \textbf{96} of {\em Graduate Studies in Mathematics}.
\newblock American Mathematical Society, Providence, RI, 2008.

\bibitem{Kwong}
M.K. Kwong.
\newblock Uniqueness of positive solutions of {$\Delta u-u+u^p=0$} in
  {$\mathbb{R}^N$}.
\newblock {\em Arch. Rational Mech. Anal.}, \textbf{105}(3):243--266, 1989.

\bibitem{PhysRevA.30.3279}
E.~W. Laedke and K.~H. Spatschek.
\newblock Stability properties of multidimensional finite-amplitude solitions.
\newblock {\em Phys. Rev. A}, \textbf{30}:3279--3288, 1984.

\bibitem{LSSJMP}
E.~W. Laedke, K.~H. Spatschek, and L.~Stenflo.
\newblock Evolution theorem for a class of perturbed envelope soliton
  solutions.
\newblock {\em J. Math. Phys.}, \textbf{24}(12):2764--2769, 1983.

\bibitem{Leis86}
R.~Leis.
\newblock {\em Initial {B}oundary {V}alue {P}roblems in {M}athematical
  {P}hysics}.
\newblock John {W}iley \& {S}ons, 1 edition, 1986.

\bibitem{Li84}
P.-L. Lions.
\newblock The concentration-compactness principle in the calculus of
  variations. {T}he locally compact case. {I}.
\newblock {\em Ann. Inst. H. Poincar\'e Anal. Non Lin\'eaire},
  \textbf{1}(2):109--145, 1984.

\bibitem{MR2515176}
C.~Miao, G.~Xu, and L.~Zhao.
\newblock Global well-posedness and scattering for the focusing energy-critical
  nonlinear {S}chr\"odinger equations of fourth order in the radial case.
\newblock {\em J. Differential Equations}, \textbf{246}(9):3715--3749, 2009.

\bibitem{natalipastor}
F.~Natali and A.~Pastor.
\newblock The fourth-order dispersive nonlinear {S}chr\"odinger equation:
  orbital stability of a standing wave.
\newblock {\em SIAM J. Appl. Dyn. Syst.}, \textbf{14}(3):1326--1347, 2015.

\bibitem{Nir-59}
L.~Nirenberg.
\newblock On elliptic partial differential equations.
\newblock {\em Ann. Scuola Norm. Sup. Pisa (3)}, \textbf{13}:115--162, 1959.

\bibitem{MR2353631}
B.~Pausader.
\newblock Global well-posedness for energy critical fourth-order
  {S}chr\"odinger equations in the radial case.
\newblock {\em Dyn. Partial Differ. Equ.}, \textbf{4}(3):197--225, 2007.

\bibitem{MR2502523}
B.~Pausader.
\newblock The cubic fourth-order {S}chr\"odinger equation.
\newblock {\em J. Funct. Anal.}, \textbf{256}(8):2473--2517, 2009.

\bibitem{MR2505703}
B.~Pausader.
\newblock The focusing energy-critical fourth-order {S}chr\"odinger equation
  with radial data.
\newblock {\em Discrete Contin. Dyn. Syst.}, \textbf{24}(4):1275--1292, 2009.

\bibitem{MR2746203}
B.~Pausader and S.~Shao.
\newblock The mass-critical fourth-order {S}chr\"odinger equation in high
  dimensions.
\newblock {\em J. Hyperbolic Differ. Equ.}, \textbf{7}(4):651--705, 2010.

\bibitem{MR3078112}
B.~Pausader and S.~Xia.
\newblock Scattering theory for the fourth-order {S}chr\"odinger equation in
  low dimensions.
\newblock {\em Nonlinearity}, \textbf{26}(8):2175--2191, 2013.

\bibitem{MR0192184}
S.~I. Poho{\v{z}}aev.
\newblock On the eigenfunctions of the equation {$\Delta u+\lambda f(u)=0$}.
\newblock {\em Dokl. Akad. Nauk SSSR}, \textbf{165}:36--39, 1965.

\bibitem{Stuart}
C.A. Stuart.
\newblock Lectures on the orbital stability of standing waves and application
  to the nonlinear {S}chr\"odinger equation.
\newblock {\em Milan J. Math.}, \textbf{76}:329--399, 2008.

\bibitem{SulemSulem}
C.~Sulem and P-L. Sulem.
\newblock {\em The {N}onlinear {S}chr\"odinger {E}quation: {S}elf-focusing and
  {W}ave {C}ollapse}, volume \textbf{139} of {\em Applied Mathematical
  Sciences}.
\newblock Springer, 1 edition, 1999.

\bibitem{Weinstein83}
M.~I. Weinstein.
\newblock Nonlinear {S}chr\"odinger {E}quations and {S}harp {I}nterpolation
  {E}stimates.
\newblock {\em Commum. Math. Phys.}, \textbf{87}:567--576, 1983.

\bibitem{MR820338}
M.I. Weinstein.
\newblock Lyapunov stability of ground states of nonlinear dispersive evolution
  equations.
\newblock {\em Comm. Pure Appl. Math.}, \textbf{39}(1):51--67, 1986.

\end{thebibliography}
\end{document}